\documentclass[12pt]{amsart}

\usepackage{amssymb,amsbsy}
\usepackage{latexsym,color,longtable}
\usepackage{verbatim,euscript}
\usepackage{fullpage,array}
\usepackage{tikz}
\usetikzlibrary{arrows,shapes,trees}

\numberwithin{equation}{section}
\usepackage[colorlinks=true,linkcolor=blue,urlcolor=violet,citecolor=magenta]{hyperref}

\definecolor{my_color}{rgb}{0,0.5,0.5}
\definecolor{mixt}{rgb}{0.5,0.3,0.2}
\definecolor{darkgreen}{rgb}{0.09, 0.35, 0.27}
\definecolor{forest}{rgb}{0.13, 0.55, 0.13}

\input {cyracc.def}
\font\tencyr=wncyr10 
\font\tencyi=wncyi10 
\font\tencysc=wncysc10 
\def\rus{\tencyr\cyracc}
\def\rusi{\tencyi\cyracc}
\def\rusc{\tencysc\cyracc}

\catcode`\@=\active
\catcode`\@=11

\renewcommand{\@cite}[2]{[{{\bf #1}\if@tempswa , #2\fi}]}
\renewcommand{\@biblabel}[1]{[{\bf #1}]\hfill}
\catcode`\@=12

\newtheorem{thm}{Theorem}[section]
\newtheorem{lm}[thm]{Lemma}
\newtheorem{cl}[thm]{Corollary}
\newtheorem{prop}[thm]{Proposition}

\theoremstyle{remark}
\newtheorem{rmk}[thm]{Remark}

\theoremstyle{definition}
\newtheorem{ex}[thm]{Example}
\newtheorem{df}{Definition}

\newcommand {\ah}{{\mathfrak a}}
\newcommand {\be}{{\mathfrak b}}
\newcommand {\ce}{{\mathfrak c}}

\newcommand {\ff}{{\mathfrak f}}
\newcommand {\g}{{\mathfrak g}}
\newcommand {\h}{{\mathfrak h}}

\newcommand {\el}{{\mathfrak l}}
\newcommand {\n}{{\mathfrak n}}

\newcommand {\p}{{\mathfrak p}}
\newcommand {\q}{{\mathfrak q}}
\newcommand {\rr}{{\mathfrak r}}
\newcommand {\es}{{\mathfrak s}}
\newcommand {\te}{{\mathfrak t}}
\newcommand {\ut}{{\mathfrak u}}

\newcommand {\slno}{{\mathfrak {sl}}_{n+1}}

\newcommand {\son}{{\mathfrak {so}}_{n}}

\newcommand {\eus}{\EuScript}

\newcommand {\gS}{{\eus S}}

\newcommand {\esi}{\varepsilon}
\newcommand {\ap}{\alpha}

\newcommand {\lb}{\lambda}

\newcommand {\ca}{{\mathcal A}}

\newcommand {\BC}{{\mathbb C}}

\newcommand {\BQ}{{\mathbb Q}}

\newcommand {\ad}{{\mathrm{ad\,}}}
\newcommand {\ads}{ {\mathrm{ad}}^* }

\newcommand {\Ad}{{\mathrm{Ad}}}

\newcommand {\ind}{{\mathrm{ind\,}}}
\newcommand {\Lie}{{\mathrm{Lie\,}}}

\newcommand {\rk}{{\mathsf{rk\,}}}

\newcommand {\supp}{{\mathsf{supp}}}

\newcommand {\trdeg}{{\mathrm{trdeg\,}}}

\newcommand {\GR}[2]{{\textrm{{\sf\bfseries #1}}}_{#2}}

\newcommand {\un}{\underline}

\newcommand {\beq}{\begin{equation}}
\newcommand {\eeq}{\end{equation}}

\newcommand {\nap}{\n_{\{\ap\}}}
\newcommand {\bb}{\boldsymbol{b}}
\newcommand {\bi}{{\boldsymbol{i}}}
\newcommand {\bxi}{\boldsymbol{\zeta}}
\newcommand {\CP}{{\sf CP}}

\newcommand{\curge}{\succcurlyeq}
\newcommand{\curle}{\preccurlyeq}
\renewcommand{\le}{\leqslant}
\renewcommand{\ge}{\geqslant}
\renewcommand{\lg}{\langle}
\newcommand{\rg}{\rangle}

\newfam\Bbbfam%
\font\Bbbfont=msbm10 scaled 1200%
\font\Bbbsmallfont=msbm8%
\textfont\Bbbfam=\Bbbfont\scriptfont\Bbbfam=\Bbbsmallfont

\begin{document}
\setlength{\parskip}{2pt plus 4pt minus 0pt}
\hfill {\scriptsize August 17, 2021} 
\vskip1ex

\title[Commutative polarisations]{Commutative
polarisations and the Kostant cascade}
\author{Dmitri I. Panyushev}
\address{
Institute for Information Transmission Problems of the R.A.S., 
Moscow 127051, Russia}
\email{panyushev@iitp.ru}
\thanks{This research was funded by RFBR, project {\rus N0} 20-01-00515.}
\dedicatory{To my friend and colleague Alexander G. Elashvili, with gratitude} 
\begin{abstract}
Let $\g$ be a complex simple Lie algebra. We classify the parabolic subalgebras $\p$ of $\g$ 
such that the nilradical of $\p$ has a commutative polarisation. The answer is given in terms of
the Kostant cascade. It requires also the notion of an optimal nilradical and some properties of abelian ideals in a Borel subalgebra of $\g$.
\end{abstract}
\maketitle

\section{Introduction}
Let $\q$ be a complex algebraic Lie algebra and $\eus S(\q)$ its symmetric algebra. The {\it index\/} of 
$\q$, denoted $\ind\q$, is the minimal dimension of the stabilisers $\q^\xi$, $\xi\in\q^*$, with respect to 
the coadjoint representation of $\q$. Then $\xi$ is said to be {\it regular}, if $\dim\q^\xi=\ind\q$. If 
$\ah\subset\q$ is an abelian subalgebra, then $\eus S(\ah)$ is a Poisson commutative subalgebra of 
$\eus S(\q)$. Hence $\dim\ah=\trdeg \eus S(\ah)\le (\dim\q+\ind\q)/2=:\bb(\q)$, see \cite[0.2]{vi90} 
or~\cite[Theorem\,14]{ooms}. If $\dim\ah=\bb(\q)$, then $\ah$ is a maximal isotropic subspace of $\q$ 
with respect to the Kirillov form associated with any regular $\xi\in\q^*$. Following~\cite[Sect.\,5]{ooms}, 
such $\ah$ is called a {\it commutative polarisation} (=\,\CP) in $\q$. 
We also say that $\q$ is a Lie algebra {\it with \CP}\ or that $\q$ {\it has a} \CP.

For any Lie algebra $\q$, there is a Poisson-commutative subalgebra $\eus A\subset \eus S(\q)$ such 
that $\trdeg \eus A= \bb(\q)$~\cite{sad}. Hence, if $\ah$ is a \CP\ in $\q$, then 
$\eus S(\ah)$ is a Poisson-commutative subalgebra of $\eus S(\q)$ with $\trdeg \eus S(\ah)=\bb(\q)$. 
Therefore, Lie algebras $\q$ having a \CP\ admit a nice explicit Poisson-commutative subalgebra of 
$\eus S(\q)$, with {\bf maximal} transcendence degree. (A generalisation of Sadetov's result to 
commutative subalgebras of the enveloping algebra $\eus U(\q)$ is recently obtained by
Yakimova~\cite{kos}.)

Basic results on commutative polarisations are obtained by Elashvili--Ooms~\cite{ag03}. They also 
proved that, for any parabolic subalgebra $\p$ of a simple Lie algebra of type $\GR{A}{n}$ or 
$\GR{C}{n}$, the nilradical of $\p$ has a \CP. In this article, we characterise the parabolic subalgebras 
$\p$ such that the nilradical $\n=\p^{\sf nil}$ has a \CP\ for any simple Lie algebra $\g$. We 
assume that $\p$ is standard, i.e., $\p\supset\be$ for a fixed Borel subalgebra $\be$ of $\g$. Then 
$\n\subset [\be,\be]$ is a $\be$-ideal. It is easily seen that if $\n$ contains an abelian subalgebra, then 
it contains an abelian $\be$-{\bf ideal} of the same dimension, cf.~\cite[Theorem\,4.1]{ag03}. If 
$\ah\lhd\n$ is an abelian $\be$-ideal with $\dim\ah=\bb(\n)$, then $\ah$ is said to be a \CP-{\it ideal} (of
$\n$). 
Unless otherwise stated, ``nilradical'' means ``the nilradical of a standard parabolic subalgebra of $\g$''. 
For any nilradical $\n$ with \CP, we also point out a \CP-ideal in $\n$. Our major tools are: 
\begin{itemize}
 \item a certain set $\eus K$ of strongly orthogonal roots in the root system $\Delta$ of $\g$ (the 
{\it Kostant cascade})~\cite{jos77,ko12}; 
 \item Joseph's formula for $\ind\n$ and the notion of an optimal parabolic subalgebra~\cite{jos77},
 which allows us to define the {\it optimisation\/} of a nilradical;   
 \item theory of abelian ideals of $\be$~\cite{imrn} and related results on root systems~\cite{p06,p20b}. 
\end{itemize}
Let $\Delta^+$ be the set of {\it positive\/} roots associated with $\be$ and $\Pi\subset\Delta^+$ the set 
of {\it simple\/} roots. Then $\eus K=\{\beta_1,\dots,\beta_m\}$ is a certain subset of $\Delta^+$ and the 
roots in $\eus K$ are {\it strongly orthogonal}, which means that $\beta_i\pm\beta_j\not\in\Delta$ for all 
$i,j$. The inductive construction of $\eus K$, which begins with the highest root $\theta=\beta_1$, makes 
$\eus K$ a graded poset such that $\beta_1$ is the unique maximal element. The partial order ``$\curle$''
in $\eus K$ is the restriction of the {\it root order\/} in $\Delta^+$. For any $\beta\in\eus K$, the upper ideal 
$\{\beta'\in \eus K\mid \beta\curle\beta'\}$ is a chain (i.e., it is linearly ordered). To each $\beta\in\eus K$, 
we naturally attach the subset $\Phi(\beta)\subset\Pi$ such that 
$\Pi=\bigsqcup_{\beta\in\eus K}\Phi(\beta)$ and $\#\Phi(\beta)\le 2$, see Section~\ref{sect:marked-optim} 
for details. This  $\Phi$ is regarded as a map from $\eus K$ to $2^{\Pi}$.
Then $\Phi^{-1}$ is a well-defined surjective map from $\Pi$
to $\eus K$. The triple $(\eus K, \curle, \Phi)$ is called the {\it marked cascade poset\/} of $\g$.

Let $\Delta(\n)\subset\Delta^+$ denote the set of roots of a nilradical $\n$. The {\it optimisation\/} of $\n$ 
is the maximal nilradical $\tilde\n$ such that $\Delta(\n)\cap\eus K=\Delta(\tilde\n)\cap\eus K$ and an 
important property of the passage $\n\leadsto\tilde\n$ is that $\bb(\n)=\bb(\tilde\n)$. (A practical
description of $\tilde\n$ is given in Section~\ref{subs:optim}.) For $\ap\in\Pi$, let $\nap$ be the nilradical 
of the corresponding maximal parabolic subalgebra. If $\nap$ is an abelian Lie algebra, then 
$\nap$ is its own \CP-ideal. The abelian nilradicals play a key role in our theory. In 
Sections~\ref{sect:CP-da} and \ref{sect:ohne-CP}, we prove that a nilradical $\n$ has a \CP\ if and only if 
at least one of the following two conditions is satisfied:
\begin{enumerate}
\item 
$\n$ is the {\it Heisenberg nilradical}, i.e., one associated with the highest root; in this case, if $\ah$ 
is any maximal abelian ideal of $\be$, then $\ah\cap\n$ is a \CP-ideal of $\n$, and vice versa.
\item 
there is an abelian nilradical $\nap$ such that $\n$ is contained in the optimisation of  $\nap$. There 
can be several abelian nilradicals with this property, and, for a ``right'' choice of such $\check\ap\in\Pi$, a
\CP-ideal of $\n$ is $\n\cap\n_{\{\check\ap\}}$.
\end{enumerate}
A complement to (2) is that, 
for $\g\ne\slno$, one can always pick $\check\ap$ in $\Phi(\Phi^{-1}(\ap))$. (See Section~\ref{sect:CP-da}
for details.)
For all simple Lie algebras $\g$, the Hasse diagrams of $(\eus K,\curle)$ and the subsets 
$\Phi(\beta)$, $\beta\in\eus K$, are depicted in Section~\ref{sect:tables}. That visual information makes
application of our main result to be easy and routine, see Example~\ref{ex:small-rk}.  
The description above shows that if $\g$ has no parabolic subalgebras with abelian nilradicals, then the 
Heisenberg nilradical is the only nilradical with \CP. This happens precisely if $\g$ is of type 
$\GR{G}{2}$, $\GR{F}{4}$, $\GR{E}{8}$. Another consequence is that $\n$ has a \CP\ if and only if its 
optimisation, $\tilde\n$, does. For $\GR{A}{n}$ and 
$\GR{C}{n}$, there is an abelian nilradical $\nap$ such that $\widetilde\nap=[\be,\be]$. This implies that
here any nilradical has a \CP, which recovers a result of Elashvili--Ooms~\cite[Section\,6]{ag03}.

In Section~\ref{sect:compl}, we prove that if $\p$ is an optimal parabolic subalgebra, then
$\ind\p+\ind \p^{\sf nil}=\rk\g$. For any nilradical $\n=\p^{\sf nil}$, Joseph constructs a solvable Lie 
algebra $\ff$ with $\n\subset\ff\subset\be$, which is used for computing $\ind\n$~\cite{jos77}, see also
Section~\ref{subs:optim}. We prove that $\ff$ is Frobenius (i.e., $\ind\ff=0$), $\bb(\ff)=\bb(\n)$, and that 
$\ff$ has a \CP\ if and only if $\n$ does. Then we provide an invariant-theoretic consequence of the
relation $\bb(\q')=\bb(\q)$ for Lie algebras $\q'\subset\q$. Namely, in this case one has 
$\gS(\q)^\q\subset \gS(\q')$, see Proposition~\ref{prop:inv-th}. Note also that if $\ah$ is a \CP\ in $\q$, 
then $\bb(\ah)=\dim\ah=\bb(\q)$.

\un{Main notation}. 
Let $G$ be a simple algebraic group with $\g=\Lie(G)$. Then
\begin{itemize}
\item[--] $\be$ is a fixed Borel subalgebra of $\g$ with $\ut^+=\ut=[\be,\be]$;
\item[--] $\te$ is a fixed Cartan subalgebra in $\be$ and $\Delta$ is the root system of $\g$ with respect to $\te$;
\item[--] $\Delta^+$ is the set positive roots corresponding to $\ut$;  
\item[--] $\Pi=\{\ap_1,\dots,\ap_n\}$ is the set of simple roots in $\Delta^+$ and the corresponding fundamental weights are $\varpi_1,\dots,\varpi_n$;
\item[--] $\te^*_\BQ$ is the $\BQ$-vector subspace of $\te^*$ spanned by $\Delta$, and $(\ ,\, )$ is the 
positive-definite form on $\te^*_\BQ$ induced by the Killing form on $\g$; as usual, 
$\gamma^\vee=2\gamma/(\gamma,\gamma)$ for $\gamma\in\Delta$.
\item[--] For each $\gamma\in\Delta$, $\g_\gamma$ is the root space in $\g$ and 
$e_\gamma\in\g_\gamma$ is a nonzero vector;
\item[--]  If $\ce\subset\ut^\pm$ is a $\te$-stable subspace, then $\Delta(\ce)\subset \Delta^\pm$ is the set of 
roots of $\ce$;
\item[--]  $\theta$ is the highest root in $\Delta^+$;
\item[--]  $\bb(\q)=(\dim\q+\ind\q)/2$ for a Lie algebra $\q$;
\item[--]  In explicit examples, the Vinberg--Onishchik numbering of simple roots is used, 
see~\cite[Table\,1]{VO}.
\end{itemize}

\section{The marked cascade poset and optimisation} 
\label{sect:marked-optim}

\subsection{The root order in $\Delta^+$ and commutative roots}
\label{subs:root-order}
We identify $\Pi$ with the vertices of the Dynkin diagram of $\g$. For any $\gamma\in\Delta^+$, let 
$[\gamma:\ap]$ be the coefficient of $\ap\in\Pi$ in the expression of $\gamma$ via $\Pi$. The 
{\it support\/} of $\gamma$ is $\supp(\gamma)=\{\ap\in\Pi\mid [\gamma:\ap]\ne 0\}$. As is well known, 
$\supp(\gamma)$ is a connected subset of the Dynkin diagram. For instance, $\supp(\theta)=\Pi$
and $\supp(\ap)=\{\ap\}$. Let ``$\curle$'' denote the {\it root order\/} in $\Delta^+$, i.e., we write 
$\gamma\curle\gamma'$ if $[\gamma:\ap]\le [\gamma':\ap]$ for all $\ap\in\Pi$. Then $\gamma'$ covers
$\gamma$ if and only if $\gamma'-\gamma\in\Pi$, which implies that $(\Delta^+,\curle)$ is a graded 
poset. Write $\gamma\prec\gamma'$ if $\gamma\curle\gamma'$ and $\gamma\ne\gamma'$. 
\\ \indent
An {\it upper ideal\/} of $(\Delta^+,\curle)$ is a subset $I$ such that if $\gamma\in I$ and 
$\gamma\prec\gamma'$, then $\gamma'\in I$. Therefore, $I$ is an upper ideal if and only if 
$\ce=\bigoplus_{\gamma\in I} \g_\gamma$ is a $\be$-ideal of $\ut$ (i.e., $[\be,\ce]\subset\ce$). 
For $\gamma\in\Delta^+$, let $I\lg\gamma\rg$ denote the upper ideal with the unique minimal element 
$\gamma$. That is, $I\lg\gamma\rg=\{\nu\in\Delta^+\mid \nu\curge\gamma\}$. Let 
$\ce\lg\gamma\rg$ be the corresponding $\be$-ideal in $\ut$. As in~\cite{p06}, $\gamma$ is said to 
be {\it commutative}, if $\ce\lg\gamma\rg$ is an abelian ideal. Let $\Delta^+_{\sf com}$ denote the set of 
commutative roots. A description of $\Delta^+_{\sf com}$ for all simple Lie algebras (=\,reduced 
irreducible root systems), which we recall below, is given in~\cite[Theorem\,4.4]{p06}. Write 
$\theta=\sum_{i=1}^n m_i\ap_i$ (i.e., $m_i=[\theta:\ap_i]$) and consider the element 
$[\theta/2]:=\sum_{i=1}^n [m_i/2]\ap_i$ in the root lattice. Then $[\theta/2]\in\Delta^+\cup\{0\}$ and
for $\gamma\in\Delta^+$, one has
\beq            \label{eq:non-com}
   \gamma\not\in\Delta^+_{\sf com} \ \Longleftrightarrow \ \gamma\curle [\theta/2] .
\eeq
Note that $[\theta/2]=0$ if and only if $\g$ is of type $\GR{A}{n}$ and then $\Delta^+=\Delta^+_{\sf com}$.
In the other cases, $[\theta/2]$ is the unique maximal non-commutative root and $\Delta^+_{\sf com}$ is a proper subset of $\Delta$.
A conceptual proof for these observations appears in~\cite[Section\,4]{p20a}.
 
Given a dominant $\lb\in\te^*_\BQ$, set 
$\Delta^\pm_\lb=\{\gamma\in\Delta^\pm\mid (\lb,\gamma)=0 \}$ and
$\Delta_\lb=\Delta^+_\lb\cup \Delta^-_\lb$. Then $\Delta_\lb$ is the root system of a semisimple 
subalgebra $\g_\lb\subset \g$ and $\Pi_\lb=\Pi\cap\Delta^+_\lb$ is the set of simple roots
in $\Delta^+_\lb$. Set $\Delta_\lb^{{>}0}=\{\gamma\in \Delta^+\mid (\lb,\gamma)>0\}$. Then
\begin{itemize}
\item  $\Delta^+ =\Delta^+_\lb \sqcup \Delta_\lb^{{>}0}$;
\item $\p_\lb=\g_\lb+\be$ is a standard parabolic subalgebra of $\g$; 
\item the set of roots for the nilradical $\n_\lb=\p_\lb^{\sf nil}$ is $\Delta_\lb^{{>}0}$; it is also denoted by
$\Delta(\n_\lb)$.
\end{itemize}
If $\lb=\theta$, then  $\n_\theta$ is a {\it Heisenberg Lie algebra}~\cite[Sect.\,2]{jos76}. In this special case, 
we write $\eus H_\theta$ in place of $\Delta_\theta^{{>}0}$, 
and $\eus H_\theta=\Delta(\n_\theta)$ is said to be the {\it Heisenberg subset\/} (of $\Delta^+$).

\subsection{The cascade poset}
\label{subs:MCP} 
The construction of the Kostant cascade $\eus K$ is briefly exposed below, while the whole story can be found in~\cite[Sect.\,2]{jos77}, \cite[Section\,3a]{lw}, and \cite{ko12}. Whenever we wish to stress that
$\eus K$ is associated with $\g$, we write $\eus K(\g)$ for it.
\\ \indent
{\it\bfseries (1)}  We begin with $(\g\lg1\rg, \Delta(1),\beta_1)=(\g,\Delta,\theta)$ and consider the (possibly 
reducible) root system $\Delta_\theta$. The highest root $\theta=\beta_1$ is the unique element of the 
{\bf first} (highest) level in $\eus K$. Let $\Delta_\theta=\bigsqcup_{j=2}^{d_2} \Delta\lg j\rg$
be the decomposition into irreducible root systems and $\Pi\lg j\rg=\Pi\cap \Delta\lg j\rg$. Then
$\Pi_\theta=\bigsqcup_{j=2}^{d_2}\Pi\lg j\rg$ and $\{\Pi\lg j\rg\}$ are the connected components of
$\Pi_\theta\subset \Pi$.
\\ \indent
{\it\bfseries (2)}   Write $\g\lg j\rg$ for the simple subalgebra of $\g$ with root system $\Delta\lg j\rg$. Then
$\g_\theta=\bigoplus_{j=2}^{d_2}\g\lg j\rg$.  Let $\beta_{j}$ be the highest root in $\Delta\lg j\rg^{+}=\Delta\lg j\rg\cap\Delta^+$. The roots $\beta_2,\dots,\beta_{d_2}$ are 
the {\it descendants\/} of $\beta_1$, and they form the {\bf second} level of $\eus K$.
Note that $\supp(\beta_j)=\Pi\lg j\rg$, hence different descendants have disjoint supports. 
\\ \indent
{\it\bfseries (3)}   Making the same step with each pair $(\Delta\lg j\rg,\beta_j)$, $j=2,\dots,d_2$, we get a collection 
of smaller simple subalgebras inside each $\g\lg j\rg$ and smaller irreducible root systems inside $\Delta\lg j\rg$. 
This provides the descendants for each $\beta_j$ ($j=2,\dots,d_2$), i.e., the elements of the {\bf third} 
level in $\eus K$. And so on...
\\ \indent
{\it\bfseries (4)}   The procedure eventually terminates and yields a maximal set 
$\eus K=\{\beta_1,\beta_2,\dots,\beta_m\}$ of strongly orthogonal roots in $\Delta^+$, which is called 
the {\it Kostant cascade}. 

\noindent
Thus, each $\beta_i\in\eus K$ occurs as the highest root of a certain irreducible root system $\Delta\lg i\rg$ 
inside $\Delta$ such that $\Pi\lg i\rg=\Pi\cap\Delta\lg i\rg^+$ is a basis for $\Delta\lg i\rg$. 

We think of $\eus K$ as poset such that $\beta_1=\theta$ is the unique maximal element and each 
$\beta_i$ covers exactly its own descendants. If $\beta_j$ is a descendant of $\beta_i$, then 
$\beta_j\prec \beta_i$ in $(\Delta^+,\curle)$
and $\supp(\beta_j)\varsubsetneq\supp(\beta_i)$, while different descendants of 
$\beta_i$ are not comparable in $\Delta^+$. Therefore the poset structure of $\eus K$ is the restriction of the  root order in $\Delta^+$.
The resulting poset $(\eus K, \curle)$ is called the {\it cascade poset}. The numbering of $\eus K$ is not 
canonical. We only assume that it is a linear extension of $(\eus K,\curle)$, i.e., if $\beta_j$ is a 
descendant of $\beta_i$, then $j>i$. 

Using the decomposition $\Delta^+ =\Delta^+_\theta \sqcup \eus H_\theta$ and induction on $\rk\g$, one
readily obtains the disjoint union determined by $\eus K$: 
\beq            \label{eq:decomp-Delta}
     \Delta^+=\bigsqcup_{i=1}^m \eus H_{\beta_i} ,
\eeq
where $\eus H_{\beta_i}$ is the Heisenberg subset in $\Delta\lg i\rg^+$ and 
$\eus H_{\beta_1}=\eus H_\theta$. The geometric counterpart of this decomposition is the direct sum of 
vector spaces
\[
      \ut=\bigoplus_{i=1}^m  \h_i ,
\]
where $\h_i$ is the Heisenberg Lie algebra in $\g\lg i\rg$ and $\Delta(\h_i)=\eus H_{\beta_i}$. In 
particular, $\h_1=\n_\theta$. For any $\beta_i\in\eus K$, we set $\Phi(\beta_i)=\Pi\cap \eus H_{\beta_i}$.
It then follows from~\eqref{eq:decomp-Delta} that 
\[
           \Pi=\bigsqcup_{\beta_i\in\eus K} \Phi(\beta_i) .
\]
One can think of $\Phi$ as a map from $\eus K$ to $2^{\Pi}$. Note that $\#\Phi(\beta_i)\le 2$ and
$\#\Phi(\beta_i)= 2$ if and only if the root system $\Delta\lg i\rg$ is of type $\GR{A}{n}$ with $n\ge 2$;
equivalently, $\Pi\lg i\rg$ is a non-trivial chain in the Dynkin diagram and all roots in $\Pi\lg i\rg$ have the same 
length. Our definition of the subsets $\Phi(\beta_i)$ yields the well-defined map  $\Phi^{-1}: \Pi\to \eus K$, 
where $\Phi^{-1}(\ap)=\beta_i$ if $\ap\in \Phi(\beta_i)$. Note that $\ap\in\Phi(\Phi^{-1}(\ap))$.

\begin{df}
The cascade poset $(\eus K,\curle)$ with subsets $\Phi(\beta_i)$ attached to the respective nodes is
said to be the {\it marked cascade poset} ({\sf MCP}).  We think of {\sf MCP} as a triple 
$(\eus K, \curle, \Phi)$.
\end{df}

The Hasse diagrams of all cascade posets, with subsets $\Phi(\beta_i)$ attached to each node, are 
presented in Section~\ref{sect:tables}. These diagrams (without $\Phi$-data) appear already 
in~\cite[Section\,2]{jos77}. Note that Joseph uses the reverse order in $\eus K$ and our upper ideals of 
$(\eus K,\curle)$ are ``parabolic subsets'' in Joseph's terminology.
\\  \indent
Let us gather some properties of {\sf MCP} that either are already explained above or easily follow from the construction of $\eus K$.

\begin{lm}    \label{lm:K-svojstva}   
Let $(\eus K, \curle, \Phi)$ be a {\sf MCP}.
\begin{enumerate}
\item The partial order ``$\curle$'' coincides with the restriction to $\eus K$ of the root order in 
$\Delta^+$;
\item two elements $\beta_i,\beta_j\in\eus K$ are comparable if and only if\/ 
$\supp(\beta_i)\cap\supp(\beta_j)\ne \varnothing$; and then one support is contained in another;
\item each $\beta_j$, $j\ge 2$, is covered by a unique element of $\eus K$;
\item for any $\beta_j\in\eus K$, the interval 
$[\beta_j,\beta_1]_{\eus K}=\{\nu\in\eus K\mid \beta_j\curle\nu\curle\beta_1\} \subset\eus K$ is a chain.
\item  For $\ap\in\Pi$, we have
$\ap\in\Phi(\beta_i)$ if and only if $(\ap, \beta_i)>0$.
\end{enumerate}
\end{lm}
Clearly, $\#\eus K\le \rk\g$ and  the equality holds if and only if each $\beta_i$ is a multiple of a fundamental weight for $\g\lg i\rg$. Recall that $\theta$ is a multiple of a fundamental weight of $\g$ if and only if $\g$ is not of type $\GR{A}{n}$, $n\ge 2$.
It is well known that the following conditions are equivalent:
{\sl (1)} $\ind\be=0$; {\sl (2)} $\#\eus K=\#\Pi$, see e.g.~\cite[Prop.\,4.2]{ap97}. 
This happens exactly if $\g\not\in\{\GR{A}{n}, \GR{D}{2n+1},
\GR{E}{6}\}$. In this case, $\Phi^{-1}$ yields a bijection between $\eus K$ and $\Pi$.

\subsection{Optimal parabolic subalgebras and nilradicals}
\label{subs:optim}
Let $\p\supset\be$ be a standard parabolic subalgebra of $\g$, with nilradical $\n=\p^{\sf nil}$.
Then $\Delta(\n)$ stands for the set of (positive) roots corresponding to $\n$.
If $\Pi\cap\Delta(\n)=T$, then we also write $\n=\n_T$ and $\p=\p_T$. In this case,
$\Pi\setminus T$ is the set of simple roots for the standard Levi subalgebra of $\p_T$.
Hence $T\ne \varnothing$ if and only if $\p_T$ is a proper parabolic subalgebra of $\g$. 

We will be interested in the elements of the Kostant cascade $\eus K$ contained in $\Delta(\n_T)$.
Set $\eus K_T=\eus K\cap \Delta(\n_T)$. Clearly, $\theta=\beta_1\in \eus K_T$ for any proper parabolic 
subalgebra.

\begin{lm}    \label{lm:optim}
For any $T\subset\Pi$, one has 
\begin{enumerate}
\item $\eus K_T$ is an upper ideal of\/ $(\eus K,\curle)$;
\item $T\subset\bigcup_{\beta_j\in \eus K_T} \Phi(\beta_j)$ and\/ 
$\n_T\subset\bigoplus_{\beta_j\in \eus K_T} \h_{j}$.
\end{enumerate}
\end{lm}
\begin{proof}
Use Lemma~\ref{lm:K-svojstva}(1) and the fact that $\Delta(\n_T)$ is an upper ideal of 
$(\Delta^+,\curle)$. 
\end{proof}
\noindent
Following Joseph~\cite[4.10]{jos77}, a standard parabolic subalgebra $\p_T$  
is said to be {\it optimal} if, in our notation,
\[
     T=\bigcup_{\beta_j\in \eus K_T} \Phi(\beta_j).
\]
We also apply this term to $\n_T$. Clearly, $\n_T$ is optimal if and only if 
$\n_T=\bigoplus_{\beta_j\in \eus K_T} \h_{j}$. This also suggests the following construction. Given any 
$T\subset\Pi$, set $\tilde T=\bigcup_{\beta_j\in \eus K_T} \Phi(\beta_j)$ and consider the nilradical 
$\n_{\tilde T}$. Then $\eus K_T=\eus K_{\tilde T}$ and 
$\n_T\subset\n_{\tilde T}=\bigoplus_{\beta_j\in \eus K_T} \h_{i}$. Hence $\n_{\tilde T}$ is optimal, 
it is the minimal optimal nilradical containing $\n_T$, and it is the maximal element of the set
of nilradicals $\{\n'\mid \Delta(\n')\cap\eus K=\eus K_T\}$.

\begin{df}
The nilradical $\n_{\tilde T}$ is called the {\it optimisation} of $\n_T$. 
\end{df}
\noindent 
If $T\subset \Pi$ is not specified for a given $\n$, then we write $\eus K(\n):=\eus K\cap\Delta(\n)$ and 
$\tilde\n$ stands for the optimisation of $\n$. The optimal nilradical $\n_{\tilde T}$ is also denoted by 
$\widetilde{\n_T}$. The benefit of optimisation is that the passage from $\n$ to $\tilde\n$  does not 
change the upper ideal $\eus K(\n)$ and the number $\bb(\n)$.
There is also a simple relation between $\ind\n$ and $\ind{\tilde\n}$. The following is 
essentially contained in~\cite{jos77}.

\begin{prop}   \label{prop:jos}
Let\/ $\tilde\n$ be the optimisation of a nilradical\/ $\n$. Then 
\[
     \dim\n +\ind\n=\dim\tilde\n + \# \eus K(\n) .
\]
\end{prop}
\begin{proof}
In place of our $\tilde\n$, Joseph works with the solvable ideal $\ff=\ff_\n\lhd \be$ that is the sum of
$\tilde\n$ and the $\BC$-linear span of $\{[e_\beta,e_{-\beta}] \mid \beta\in \eus K(\n)\}$ in $\te$.
And it is proved in \cite[Proposition\,2.6]{jos77} that $\dim\ff=\dim\n+\ind\n$.
\end{proof}
\begin{cl}     \label{cor:jos}
$\ind\tilde\n=  \# \eus K(\tilde\n)= \# \eus K(\n)$ and 
$\bb(\n)=\bb(\tilde\n)$.
\end{cl}
\begin{proof}
For $\tilde\n=\n$,  we obtain $\ind\tilde\n=\# (\eus K(\n))$. Then the formula Proposition~\ref{prop:jos}
reads $\dim\n+\ind\n=\dim\tilde\n+\ind\tilde\n$.
\end{proof}

\begin{ex}          \label{ex:sln}
(1) Let $\g=\slno$ with $\Delta^+=\{\esi_i-\esi_j\mid 1\le i <j \le n+1\}$ and $\ap_i=\esi_i-\esi_{i+1}$ for
$i=1,2,\dots,n$. Then $\theta=\varpi_1+\varpi_n=\esi_1-\esi_{n+1}$ and 
$\eus K=\{\beta_1,\dots,\beta_{[n+1/2]}\}$, where $\beta_i=\esi_i-\esi_{n+2-i}$. 
Here $(\eus K\curle)$ is a chain, i.e., $\beta_i\succ \beta_{i+1}$ for all $i$,  and 
$\Phi(\beta_i)=\{\ap_i, \ap_{n+1-i}\}$ (cf. Figure~\ref{fig:An} in Section~\ref{sect:tables}). 

(2) Take $T=\{\ap_2,\ap_6\}$ and the  
nilradical $\n_T\subset \mathfrak{sl}_7$. Then
$\eus K_T=\{\beta_1,\beta_2\}$ and therefore $\tilde T=\{\ap_1,\ap_2,\ap_5,\ap_6\}$, cf. the 
matrices below.
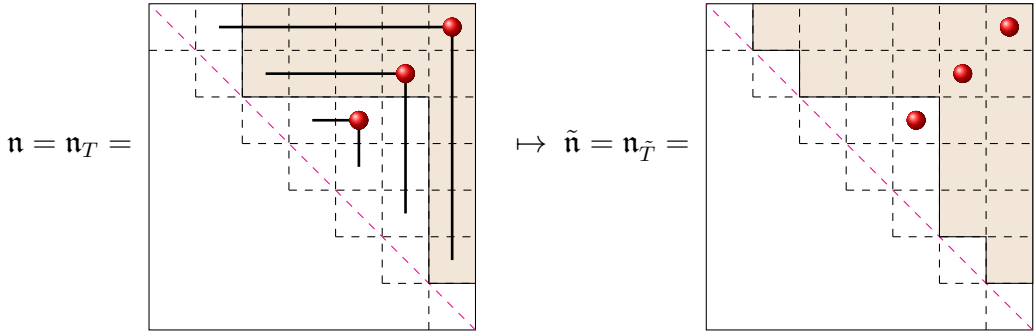
\begin{center}
\raisebox{12ex}{$\n=\n_T=$} \ 
\begin{tikzpicture}[scale= .62]
\draw (0,0)  rectangle (7,7);
\draw[dashed,magenta]  (7,0) -- (0,7);
\path[draw,fill=brown!20]  (2,7) -- (7,7) -- (7,1) -- (6,1) -- (6,5)--(2,5)--cycle ;

\path[draw, line width=1pt]  (1.5,6.5)--(6.4,6.5);
\path[draw, line width=1pt]  (6.5,1.5)--(6.5,6.4);

\path[draw, line width=1pt]  (2.5,5.5)--(5.4,5.5);
\path[draw, line width=1pt]  (5.5,2.5)--(5.5,5.4);

\path[draw, line width=1pt]  (3.5,4.5)--(4.4,4.5);
\path[draw, line width=1pt]  (4.5,3.5)--(4.5,4.4);

\foreach \x in {9,11,13}  \shade[ball color=red] (\x/2,\x/2) circle (2mm);
\path[draw,dashed]  (0,6) -- (7,6); 
\path[draw,dashed]  (1,5) -- (7,5); 
\path[draw,dashed]  (2,4) -- (7,4); 
\path[draw,dashed]  (3,3) -- (7,3); 
\path[draw,dashed]  (4,2) -- (7,2); 
\path[draw,dashed]  (5,1) -- (7,1); 

\path[draw,dashed]  (6,0) -- (6,7); 
\path[draw,dashed]  (5,1) -- (5,7); 
\path[draw,dashed]  (4,2) -- (4,7); 
\path[draw,dashed]  (3,3) -- (3,7); 
\path[draw,dashed]  (2,4) -- (2,7); 
\path[draw,dashed]  (1,5) -- (1,7); 
\end{tikzpicture}
\quad \raisebox{12ex}{$\mapsto \ \tilde\n= \n_{\tilde T}=$} \ 
\begin{tikzpicture}[scale= .62]
\draw (0,0)  rectangle (7,7);
\draw[dashed,magenta]  (7,0) -- (0,7);
\path[draw,fill=brown!20]  (2,6)--(1,6)--(1,7)--(7,7)--(7,1) -- (6,1) -- (6,2)--(5,2)--(5,5)--(2,5)--cycle ;

\foreach \x in {9,11,13}  \shade[ball color=red] (\x/2,\x/2) circle (2mm);
\path[draw,dashed]  (0,6) -- (7,6); 
\path[draw,dashed]  (1,5) -- (7,5); 
\path[draw,dashed]  (2,4) -- (7,4); 
\path[draw,dashed]  (3,3) -- (7,3); 
\path[draw,dashed]  (4,2) -- (7,2); 
\path[draw,dashed]  (5,1) -- (7,1); 

\path[draw,dashed]  (6,0) -- (6,7); 
\path[draw,dashed]  (5,1) -- (5,7); 
\path[draw,dashed]  (4,2) -- (4,7); 
\path[draw,dashed]  (3,3) -- (3,7); 
\path[draw,dashed]  (2,4) -- (2,7); 
\path[draw,dashed]  (1,5) -- (1,7); 
\end{tikzpicture}
\end{center}

\noindent
The cells with ball represent the cascade and the thick lines depict the Heisenberg subset attached to an 
element of the cascade. By Proposition~\ref{prop:jos}, we have $\ind\tilde\n=2$, $\ind\n=6$,  and
$\bb(\n)=\bb(\tilde\n)=10$. Then $\n_{\{\ap_2\}}$, $\tilde\n\cap\n_{\{\ap_3\}}$, $\tilde\n\cap\n_{\{\ap_4\}}$,
and $\n_{\{\ap_5\}}$  are the \CP-ideals for $\tilde\n$, while the only \CP-ideal for $\n$ is $\n_{\{\ap_2\}}$. 

(3) If $T=\{\ap_k\}$ for $\g=\slno$ and $k\le (n+1)/2$, then 
$\eus K\cap\Delta(\n_{\{\ap_k\}})=\{\beta_1,\dots,\beta_k\}$. Hence
$\tilde T=\{\ap_1,\dots,\ap_k,\ap_{n+1-k},\dots,\ap_n\}$. Since $\n_{\{\ap_k\}}$ is abelian, one has
$\ind\n_{\{\ap_k\}}=\dim\n_{\{\ap_k\}}=k(n+1-k)$, which comply with Proposition~\ref{prop:jos}.
\end{ex}

\begin{ex}   \label{ex:E6}
Let $\g$ be of type $\GR{E}{6}$ and $T=\{\ap_2\}$. Then $\n_T$ is the nilradical of a parabolic subalgebra $\p_T$ whose Levi subalgebra is of semisimple type $\GR{A}{1}+\GR{A}{4}$. Using data
presented in Section~\ref{sect:tables} (especially Figure~\ref{fig:En}), we see that 
$\eus K_T=\{\beta_1,\beta_2,\beta_3\}$ and $\tilde T=\Pi\setminus\{\ap_3\}$. Therefore,
\[
  \dim\n_T=25, \ \dim\n_{\tilde T}=35, \ \ind \n_{\tilde T}=3, \ \ind\n_T=13.
\]
Our main result below implies that $\n_T$ and $\n_{\tilde T}$ do not have a \CP. In fact, 
explicit computations with $\Delta(\n_T)$ show that the maximal dimension of an abelian $\be$-ideal
in $\n$ equals 15, while $\bb(\n_T)=\bb(\n_{\tilde T})=19$.
\end{ex}
\begin{rmk}   \label{rem:reduct-g}
The cascade can be defined for an arbitrary reductive algebraic Lie algebra $\h$. Namely,
$\eus K(\h)=\eus K([\h,\h])$ and if $[\h,\h]=\g_1\dotplus\dots\dotplus\g_s$ (the sum of simple ideals),
then $\eus K(\h)=\bigsqcup_{i=1}^s \eus K(\g_i)$. One may encounter this situation after the first
step in constructing $\eus K(\g)$ for a simple $\g$. Here $\eus K(\g)\setminus\{\theta\}=\eus K(\g_\theta)$ 
is the cascade for the (possibly non-simple) Lie algebra $\g_\theta\subset \g$. 
This general notion allows to provide another interpretation of optimality.
Let $\el$ be the standard Levi subalgebra of a standard parabolic subalgebra $\p\subset\g$. Then 
$\p$ is optimal if and only if $\eus K(\el)\subset \eus K(\g)$.
\end{rmk}

\section{On nilradicals with commutative polarisation}
\label{sect:CP-da}

\noindent
In this section, we describe a class of nilradicals that have a \CP \ and point out a \CP-ideal for each nilradical in this class. Afterwards, we prove in Section~\ref{sect:ohne-CP} that this class contains actually all nilradicals with \CP.

First, we observe that any simple Lie algebra $\g$ has a special parabolic subalgebra whose nilradical 
has a \CP. (For some $\g$, it appears to be the only nilradical with \CP, but this will become clear in
Section~\ref{sect:ohne-CP}.) Recall that $\n_\theta$ is a Heisenberg Lie algebra, which is also denoted 
$\h_1$ in the notation related to the cascade $\eus K$. We say that $\n_\theta$ is the {\it Heisenberg nilradical} (associated with $\g$). Here 
\[
    \Delta(\n_\theta)=\eus H_\theta=\{\gamma\in\Delta^+\mid (\gamma,\theta^\vee)=1\}\cup\{\theta\} .
\]
Since $\n_\theta$ is optimal and $\eus K\cap\eus H_\theta=\{\theta\}$, it follows from 
Corollary~\ref{cor:jos} that $\ind\n_\theta=1$. (Of course, this is easy to prove directly!) Another 
standard fact is that $\dim\n_\theta=2\mathsf{h}^*-3$, where $\mathsf{h}^*=\mathsf{h}^*(\g)$ is the 
{\it dual Coxeter number} of $\g$, see e.g.~\cite[Prop.\,1.1]{ko12}. Therefore 
$\bb(\n_\theta)=\mathsf{h}^*-1$. We say that $\gamma\in\Delta$ is {\it long}, if $(\gamma,\gamma)=(\theta,\theta)$.

\begin{prop}    \label{prop:heis}
The Heisenberg nilradical has a \CP, and the number of\/ \CP-ideals in $\n_\theta$ equals the number of long simple roots in $\Pi$. 
\end{prop}
\begin{proof}
1) The Heisenberg Lie algebra $\n_\theta$ has a basis $x_1,\dots,x_k,y_1,\dots,y_k,z$ such that 
$k=\mathsf{h}^*-2$ and the only nonzero brackets are $[x_i,y_i]=z$. Then $\lg y_1,\dots,y_k,z\rg$ is 
an abelian ideal of dimension $\bb(\n_\theta)$ and there are plenty of such ideals. But the number of 
$\be$-ideals in $\n_\theta$ is finite and depends on the 
ambient Lie algebra $\g$.
\\
2) The dimension of a \CP-ideal in $\n_\theta$ is $\mathsf{h}^*-1$.
By~\cite[Lemma\,2.2]{p20b}, if $\ce\subset\n_\theta$ is a $\be$-ideal and $\dim\ce\le \mathsf{h}^*-1$,
then $\ce$ is abelian.  On the other hand, there is a well-developed theory of abelian $\be$-ideals in
$\n_\theta$~\cite{imrn}. By that theory, the abelian ideals of dimension $\mathsf{h}^*-1$ bijectively 
correspond to the elements of shortest length in $W$ taking $\theta$ to a simple root (necessarily long). 
\end{proof}

\begin{rmk}    \label{rem:napomina-abelian}
(1) \ Let us provide an explicit construction of the \CP-ideals in $\n_\theta$. Let $\Pi_l$ be the set of
long simple roots. If $\ap\in\Pi_l$, then there is a unique element of minimal length $w_{\ap}\in W$
such that $w_{\ap}(\theta)=\ap$~\cite[Theorem\,4.1]{imrn}. Then $\ell(w_{\ap})=\mathsf{h}^*-2$ and the \CP-ideal $\ah_\ap$ 
corresponding to $\ap$ (that is, the set of roots of $\ah_\ap$) is
\[
     \Delta(\ah_\ap)=\{\theta\}\cup\{\theta-\gamma\mid \gamma\in \eus N(w_{\ap})\},
\]
see~\cite[Lemma\,1.1]{mics}. Here $\eus N(w_{\ap}):=\{\nu\in\Delta^+\mid w_{\ap}(\nu)\in -\Delta^+\}$ is the {\it inversion set} of 
$w_\ap$.

(2) \ A related description of these \CP-ideals is the following. By~\cite[Corollary\,3.8]{imrn}, the maximal abelian ideals of $\be$ are parametrised by $\Pi_l$. If $\ah_{\ap,\sf max}\lhd\be$ denotes the maximal abelian ideal corresponding to $\ap\in\Pi_l$, then the related \CP-ideal of $\n_\theta$ is   
$\ah_\ap=\n_\theta\cap \ah_{\ap,\sf max}$.
\end{rmk}

For $\ap\in\Pi$, the nilradical $\nap$ is equal to the $\be$-ideal $\ce\lg\ap\rg$ defined in 
Section~\ref{subs:root-order}. Hence $\ap\in\Pi$ is commutative if and only if the nilradical $\nap$ is 
abelian. Set $\Pi_{\sf com}=\Pi\cap\Delta^+_{\sf com}$. It is easily seen that $\ap\in\Pi_{\sf com}$ if 
and only if $[\theta:\ap]=1$. Therefore, $\Pi_{\sf com}=\Pi$ for $\GR{A}{n}$, while 
$\Pi_{\sf com}=\varnothing$ for $\g$ of type $\GR{G}{2}$, $\GR{F}{4}$, and 
$\GR{E}{8}$. If $\ap\in\Pi_{\sf com}$, then it is always long and  $\nap$ is the maximal abelian ideal 
of $\be$ that corresponds to $\ap$. In particular, if $\g$ is of type $\GR{A}{n}$, then all maximal abelian 
ideals of $\be$ are the nilradicals $\nap$, $\ap\in\Pi$.

\begin{prop}      \label{prop:1}
Let $\n$ be a nilradical. \\  \indent
{\sf 1)} \  If\/ $\n\subset\widetilde{\nap}$ for some $\ap\in\Pi$, then the upper ideal $\eus K(\n)$ is a 
chain in $(\eus K,\curle)$. 
\\ \indent
{\sf 2)} \ If\/ $\n$ is optimal and $\n\subset\widetilde{\nap}$ for some $\ap\in\Pi_{\sf com}$, then $\n$ has 
a\/ \CP. More precisely, $\n\cap \n_{\{\ap\}}$ is a \CP-ideal in $\n$. In particular, $\nap$ is a \CP-ideal in\/ $\widetilde{\nap}$.
\end{prop}
\begin{proof} 
{\sf 1)} Recall that $\eus K(\n)=\eus K\cap\Delta(\n)$ and $\eus K(\n)=\eus K(\tilde\n)$.
Set $\beta=\Phi^{-1}(\ap)$. Then 
\[
    \eus K(\n_{\{\ap\}})=\{\beta_i\in\eus K\mid\beta_i\curge\beta\}=[\beta,\beta_1]_{\eus K}
\] 
is the chain connecting $\beta$ and $\theta=\beta_1$ in $\eus K$ (cf. Lemma~\ref{lm:K-svojstva}(4)). 
Therefore, $\eus K(\n)=[\beta_k,\beta_1]_{\eus K}$ is also a (possibly shorter) chain for some
$\beta_k$ between $\beta$ and $\beta_1$. 

{\sf 2)} Let $\eus K(\n)=[\beta_k,\beta_1]_{\eus K}$. 
For simplify notation, we arrange the numbering of $\eus K$ such that 
$[\beta_k,\beta_1]_{\eus K}=\{\beta_1,\beta_2,\dots,\beta_k\}$. Since $\n$ is optimal, we have 
$\n=\bigoplus_{i=1}^k \h_i$  and $\beta_k,\dots,\beta_1\in \Delta(\nap)$. Hence $[\beta_j:\ap]=1$ for 
$j=1,\dots,k$. Since $\h_j$ is a Heisenberg Lie algebra, $\eus H_{\beta_j}\setminus \{\beta_j\}$ consists of the pairs of roots of the form 
$\{\beta_j-\gamma, \gamma\}$ and hence $\dim (\nap\cap\h_j)=(\dim\h_j+1)/2$. Therefore,
\[
   \dim (\n\cap\nap)=\sum_{j=1}^k \frac{\dim\h_j+1}{2}=\frac{\dim\n+k}{2}=\bb(\n) ,
\]
where we use the fact that $\ind\n=k$, cf. Corollary~\ref{cor:jos}. Thus, $\n\cap\nap$ is a \CP-ideal in 
$\n$. Taking $\n=\widetilde{\nap}$ yields the last assertion.
\end{proof}

\begin{rmk}
For $\g=\slno$, we have $\Pi=\Pi_{\sf com}$ and $\eus K$ is a chain. Therefore, it can happen that
there are several $\ap\in\Pi$ such that $\n=\tilde\n\subset\widetilde\nap$. Then each choice of $\ap$
provides a \CP-ideal in $\tilde\n$, cf. Example~\ref{ex:sln}(2). But this is a specific feature of $\slno$.
\end{rmk}

Next step is to consider the nilradicals inside $\widetilde\nap$ that are not necessarily optimal.
Recall that, for any $\beta\in\eus K$, one has $\#\Phi(\beta)\le 2$. Therefore, 
$\# \Phi(\Phi^{-1}(\ap))\in\{1,2\}$ for any $\ap\in\Pi$. Consequently, either $\Phi(\Phi^{-1}(\ap))=\{\ap\}$ 
or $\Phi(\Phi^{-1}(\ap))=\{\ap,\ap'\}$ for some other $\ap'\in\Pi$. The second possibility occurs only for
$\GR{A}{n}, \GR{D}{2n+1},\GR{E}{6}$. Note that if $\Phi(\beta)=\{\ap,\ap'\}$, then
$\widetilde\nap=\widetilde{\n_{\{\ap'\}}}$.

\begin{lm}              \label{lm:Phi=2}
If\/ $\Phi(\beta_j)=\{\ap,\ap'\}$ for some $\beta_j\in\eus K$, then $[\theta:\ap]=[\theta:\ap']$. In particular,
if $\ap$ is commutative, then so is $\ap'$.
\end{lm}
\begin{proof}
This is easily verified case-by-case, and an {\sl a priori\/} explanation is available, too. \\
Let $s_\gamma\in W$ be the reflection for $\gamma\in\Delta^+$. Then 
$w_0=\prod_{\beta\in\eus K}s_\beta\in W$ is the longest element and $w_0(\beta)=-\beta$ for all $\beta\in\eus K$~\cite[Prop.\,1.10]{ko12}. Recall that $\beta_j$ is the highest root in the irreducible root system
$\Delta\lg j\rg\subset\Delta$. If $\#\Phi^{-1}(\beta_j)=2$, then $\Delta\lg j\rg$ is of type $\GR{A}{p}$, $p\ge 2$, and
$\ap,\ap'$ are the extreme roots in $\Pi\lg j\rg=\Pi\cap\Delta\lg j\rg$. Set $\eus K(j)=\eus K\cap\Delta\lg j\rg$, which is
the cascade in $\Delta\lg j\rg$. It is easily seen that $w_0(\Delta\lg j\rg)=\Delta\lg j\rg$ and $w_0$ acts on
$\Delta\lg j\rg$ as the longest element in the Weyl group $W(j)$. Hence $w_0(\ap)=-\ap'$, and we are done.
\end{proof}

\begin{prop}      
\label{prop:2}
Suppose that the optimisation of\/ $\n$ is $\widetilde\nap$ for some  $\ap\in\Pi_{\sf com}$. 
Then $\n$ has a \CP \ and $\n\cap\n_{\{\check\ap\}}$ is a \CP-ideal for each 
$\check\ap\in \Phi(\Phi^{-1}(\ap))\cap \Delta(\n)$. More precisely,
\begin{itemize}
\item \ if\/ $\Phi(\Phi^{-1}(\ap))=\{\ap\}$, then $\ap\in\Delta(\n)$ and\/ $\n_{\{\ap\}}$ is a \CP-ideal in $\n$;
\item \ if\/ $\Phi(\Phi^{-1}(\ap))=\{\ap,\ap'\}$, then at least one of\/
$\nap$ and\/ $\n_{\{\ap'\}}$ is a \CP-ideal in $\n$.
\end{itemize}
\end{prop}
\begin{proof}
Set $\beta=\Phi^{-1}(\ap)$. Then $\beta$ is the unique minimal element of $\eus K(\nap)$. Because 
$\tilde\n=\widetilde\nap$, the root $\beta$ is also minimal in the chain
$ \eus K(\n)=  \eus K(\widetilde\nap)=\eus K(\nap)$.
Hence $\Delta(\n)\cap\Phi(\beta)\ne\varnothing$. 

\textbullet \ If $\ap\in \Delta(\n)$, then 
$\nap\subset\n \subset \widetilde\nap$. Hence $\dim\nap=\bb(\widetilde\nap)=\bb(\n)$, where the 
first  (resp. second) equality stems from Proposition~\ref{prop:1} (resp. Corollary~\ref{cor:jos}). 

\textbullet \ If $\ap\not\in \Delta(\n)$, then  $\Phi(\beta)=\{\ap,\ap'\}$ and $\ap'\in \Delta(\n)$ 
(otherwise, $\beta$ is not contained in $\Delta(\n)$). By Lemma~\ref{lm:Phi=2}, we have 
$\ap'\in\Pi_{\sf com}$. Then
the preceding argument works with $\ap'$ in place of  $\ap$. 
\end{proof}

\begin{prop}       \label{prop:chain-&-An}  
\leavevmode\par
\begin{enumerate}
\item  If\/ $(\eus K,\curle)$ is a chain, then, for any nilradical\/ $\n$, there is $\ap\in\Pi$ such that
$\tilde\n=\widetilde\nap$;
\item If\/ $\g=\slno$, then $\eus K$ is a chain and every nilradical\/ $\n$ has a \CP. 
\end{enumerate}
\end{prop}
\begin{proof}
(1)  Let $\beta$ be the unique minimal element of $\eus K(\n)$. Then $\Delta(\n)\cap\Phi(\beta)\ne\varnothing$, and if $\ap\in \Delta(\n)\cap\Phi(\beta)$, then $\tilde\n=\widetilde\nap$.

(2) Here $\eus K$ is a chain (cf. Figure~\ref{fig:An}) and {\bf all} $\ap\in\Pi$ are commutative.
Therefore, the assertion follows from part~(1) and Proposition~\ref{prop:2}.
\end{proof}

\begin{ex}   \label{ex:sl-2} 
Let $\g=\slno$ and  $T=\{\ap_{i_1}, \dots, \ap_{i_k} \}$. In the notation of Example~\ref{ex:sln}(1),
the minimal element of $\eus K_T=\eus K\cap \Delta(\n_T)$ is $\beta_{\boldsymbol{i}}$, where 
$\bi\le [(n+1)/2]$ is defined by condition
\beq      \label{eq:nomer-i}
        \min_{j=1,\dots,k} \left\vert \frac{n+1}{2}-i_j\right\vert=\frac{n+1}{2}-\bi .
\eeq
Then at least one of $\ap_{\bi},\ap_{n+1-\bi}$ belongs to $T$ and $\ap_j\not\in T$ if
$\bi< j< n+1-\bi$. This procedure provides a simple root in $T$ that is closest to the middle of the matrix 
antidiagonal. Then $\widetilde{\n_T}=\widetilde{\n_{\{\ap_{\bi}\}}}=\widetilde{\n_{\{\ap_{n+1-\bi}\}}}$ 
and $\nap$ is a \CP-ideal in $\n_T$ for each $\ap\in \{ \ap_{\bi}, \ap_{n+1-{\bi}} \} \cap T$.
\end{ex}

Let us return to the general case. Given $\ap\in\Pi_{\sf com}$, consider  an arbitrary nilradical $\n$ 
inside $\widetilde{\nap}$. Then $\n\subset\tilde\n\subset \widetilde{\nap}$ and, by 
Proposition~\ref{prop:1}, $\tilde\n\cap\nap$ is a \CP-ideal in $\tilde\n$. We wish to prove that $\n$ also 
has a \CP. Since $\bb(\n)=\bb(\tilde\n)$, it would be sufficient to prove that 
$\n\cap\nap=\tilde\n\cap\nap$.  Such an equality may fail for an initial choice of $\ap\in\Pi_{\sf com}$, if 
$\g=\slno$ (cf. Example~\ref{ex:sln}(2), where $\n_T\subset\widetilde{\n_{\{\ap_j\}}}$ for $j=2,3,4,5$). 
However, Example~\ref{ex:sl-2} shows that, for any nilradical $\n$ in $\slno$, there is always at least one 
right choice of $\ap$ such that $\nap\subset\n\subset\widetilde{\nap}$ 
and thereby $\nap$ is a \CP\ in both $\n$ and $\tilde\n$. Indeed, if $\beta_i$ is the minimal element of 
$\eus K(\n)$, then any root in $\Phi(\beta_i)\cap \Delta(\n)$ will do.
Therefore, we exclude the $\mathfrak{sl}$-case from the following proposition.

\begin{prop}      
\label{prop:3}
Suppose that $\g\ne\slno$ and $\n\subset \widetilde{\nap}$ for some $\ap\in\Pi_{\sf com}$.  
Then there is $\check\ap\in \Phi(\Phi^{-1}(\ap))\cap \Delta(\n)$ such that
$\n\cap\n_{\{\check\ap\}}=\tilde\n\cap\n_{\{\check\ap\}}$ and the latter is a \CP-ideal in both $\n$ and $\tilde\n$.
\end{prop}
\begin{proof} 
For $\g\ne\slno$, we can also use the property that any $\ap\in\Pi_{\sf com}$ is 
an endpoint in the Dynkin diagram.

As above, $\beta=\Phi^{-1}(\ap)$ and $\eus K(\nap)=[\beta,\beta_1]_{\eus K}$ is a chain in $\eus K$.
In view of Proposition~\ref{prop:2}, we may assume that $\tilde\n \varsubsetneq \widetilde\nap$. Then 
$\eus K(\tilde\n)=[\beta_k,\beta_1]_{\eus K}$ is a shorter chain in $\eus K$ (i.e., $\beta\prec\beta_k$) 
and $\tilde\n=\bigoplus_{j=1}^k \h_j$. If $k=1$, then $\n=\tilde\n=\h_1$ (the equality 
$\n=\tilde\n$ holds, since $\g\ne\slno$ and hence $\theta$ is a multiple of a fundamental weight), i.e.,
$\n$ has a \CP\ by Prop.~\ref{prop:heis}. Hence, we may assume that $k\ge 2$.

Since $\eus K(\n)=\eus K(\tilde\n)=[\beta_k,\beta_1]_{\eus K}$, we have $\Phi(\beta_k)\cap
\Delta(\n)\ne\varnothing$ and $\n\cap\h_i\ne\{0\}$ for $1\le i\le k$.
Because $\ap\in\Pi_{\sf com}$, $\beta=\Phi^{-1}(\ap)$, $\beta\prec\beta_k$, and $\g\ne\slno$,
one can infer from Figures~\ref{fig:An}--\ref{fig:En} in Section~\ref{sect:tables} that $\Phi(\beta_k)$ is always a sole simple 
root, say $\check\ap$. This implies that 

\textbullet \quad $\n_{\{\check\ap\}}$ is the unique minimal nilradical whose optimisation is $\tilde\n$,  

\textbullet \quad the whole of $\h_k$ belongs to $\n_{\{\check\ap\}}$. 

\noindent
Without loss of generality, we may assume that $\n=\n_{\{\check\ap\}}$. Then
\[
   \n_{\{\check\ap\}}=\bigoplus_{i=1}^k (\n_{\{\check\ap\}}\cap\h_i) \subset \tilde\n=\bigoplus_{i=1}^k \h_i ,
\]
and our goal is to prove that $\n_{\{\check\ap\}}\cap\h_i\cap \nap=\h_i\cap\nap$ for all $i\le k$.
Since $\h_k\subset \n_{\{\check\ap\}}$, this holds for $i=k$. If $i<k$, then $\beta_k\prec\beta_i$ and 
hence $\check\ap\in \supp(\beta_k)\subset\supp(\beta_i)$.  

If $\gamma\in \Delta(\h_i\cap\nap)$, then $\ap\in\supp(\gamma)$ and 
$\Phi(\beta_i)\cap \supp(\gamma)\ne\varnothing$. Here $\Phi(\beta_i)$ is also a sole root, say $\ap'$;
hence $\ap'\in \supp(\gamma)$. The key observation is that, for such a string 
$\beta\prec \beta_k\prec \beta_i$ in $\eus K$, the simple root $\check\ap$ lies between 
$\ap$ and $\ap'$ in the Dynkin diagram. Hence $\check\ap\in \supp(\gamma)$ and 
$\gamma\in \Delta(\n_{\{\check\ap\}}\cap\h_i\cap \nap)$. Thus, $\n\cap\nap=\tilde\n\cap\nap$, and it is 
an abelian $\be$-ideal in both $\n$ and $\tilde\n$. Since $\dim(\tilde\n\cap\nap)=\bb(\tilde\n)$
(Proposition~\ref{prop:1}) and $\bb(\n)=\bb(\tilde\n)$ (Corollary~\ref{cor:jos}), it is a \CP-ideal in $\n$.
\end{proof}

Combining Propositions~\ref{prop:heis}, \ref{prop:1}, \ref{prop:2}, and \ref{prop:3} yields the first part our 
classification.

\begin{thm}         \label{thm:CP-da}
Let $\n$ be the nilradical of a standard parabolic subalgebra $\p\supset\be$. Then $\n$ has a \CP\ 
whenever at least one of the following two conditions holds:
\begin{itemize}
\item $\n=\n_\theta$ is the Heisenberg nilradical. In this case, $\n_\theta\cap\ah$ is a \CP-ideal  
for any maximal abelian ideal\/ $\ah\lhd\be$, and vice versa;.
\item There is an $\ap\in\Pi_{\sf com}$ such that $\n\subset\widetilde\nap$. If $\g\ne\slno$,
then $\n\cap\n_{\{\check\ap\}}$ is a \CP-ideal for $\check\ap\in \Phi(\Phi^{-1}(\ap))\cap\Delta(\n)$. For 
$\g=\slno$, one has to pick a right $\ap\in\Pi_{\sf com}$, cf. Example~\ref{ex:sl-2}.
\end{itemize} 
\end{thm}

\begin{rmk}    \label{rem:recover}
If the poset $(\eus K,\curle)$ is a chain and the minimal element $\beta_m\in\eus K$ has the property that 
$\Phi(\beta_m)\subset\Pi_{\sf com}$, then $\widetilde\nap=\ut$ for any $\ap\in \Phi(\beta_m)$. Therefore, in such a case any nilradical has a \CP. The Hasse diagrams in 
Section~\ref{sect:tables} show that these conditions are satisfied only for $\GR{A}{n}$ and
$\GR{C}{n}$. In this way, we recover a result of Elashvili--Ooms, see~\cite[Theorem\,6.2]{ag03}.
\end{rmk}

\section{On nilradicals without commutative polarisation} 
\label{sect:ohne-CP}

\noindent
In this section, we prove the converse to Theorem~\ref{thm:CP-da}.

\begin{thm}     \label{thm:CP-net}
If a nilradical $\n$ has a \CP, then $\n=\n_\theta$ or
there is an $\ap\in\Pi_{\sf com}$ such that $\n\subset\widetilde\nap$.
\end{thm}

Recall that $\eus K(\n)=\eus K\cap\Delta(\n)$ is an upper ideal of $\eus K$ and $\eus K(\widetilde\nap)$
is a chain with minimal element $\beta=\Phi^{-1}(\ap)$ for any $\ap\in\Pi$. Given $\beta\in\eus K$, we 
say that the chain $[\beta,\beta_1]_{\eus K}$ is {\it good}, if $\Phi(\beta)\subset\Pi_{\sf com}$. Therefore, 
Theorem~\ref{thm:CP-net} is equivalent to the following statement.

\begin{thm}   \label{thm:CP-shtrikh}
Suppose that a nilradical $\n$ has a \CP.  Then $\eus K(\n)=\{\beta_1\}$ or $\eus K(\n)$ is contained 
in a good chain.
\end{thm}
 
\begin{cl}            
\label{cor:only-heis}   
If\/ $\Pi_{\sf com}=\varnothing$, then $\n_\theta=\h_1$ is the only nilradical with \CP. This happens 
exactly if\/ $\g$ is of type $\GR{G}{2}$, $\GR{F}{4}$, $\GR{E}{8}$.
\end{cl}

The following simple but important observation follows from the equality $\bb(\n)=\bb(\tilde\n)$.
 
\begin{lm}        \label{lm:optim2}
If\/ a nilradical $\n$ has a \CP, then so does the optimisation $\tilde\n$. If 
$\ah$ is a \CP-ideal in $\n$, then it is also a \CP-ideal in $\tilde\n$.
\end{lm}   

Consequently, it suffices to prove that the {\bf optimal} nilradicals that are not covered by 
Theorem~\ref{thm:CP-da} do not have a \CP. For this reason, we stick below to optimal nilradicals.
In the rest of this section, we assume that $\n=\tilde\n$ and $\n\ne\n_\theta$. The latter is equivalent to 
that $\#\eus K(\n)\ge 2$. 

Our first goal is to realise what constraints on $\n$ imposes the presence of a \CP. Let 
$\eus K(\n)=\{\beta_1,\dots,\beta_k\}$ for $k\ge 2$.  Then $\n=\tilde\n=\bigoplus_{i=1}^k\h_i$, where 
$\h_1=\n_\theta$. If $\ah\lhd\n$ is an abelian ideal of $\be$, then $\ah_i=\h_i\cap\ah$ is an abelian Lie
algebra for each $i$ and hence $\dim\ah_i\le (\dim\h_i +1)/2=\bb(\h_i)$. Then 
\[
   \dim\ah\le \sum_{i=1}^k \frac{\dim\h_i +1}{2}=\frac{\dim\n+k}{2}=\bb(\n) .
\]
Therefore, if $\ah$ is a \CP-ideal, then $\dim\ah_i= (\dim\h_i +1)/2$ for all $i$. Of course, one can always
pick an abelian ideal $\ah_i\lhd \h_i$ of required dimension for each $i$.  The question is, whether it is 
possible to choose all $\ah_i$'s in a compatible way (so that the whole of $\ah$ be abelian).

If $\n\subset \widetilde\nap$ and $\ap\in\Pi_{\sf com}$, then a compatible choice is achieved by letting
$\ah_i=\h_i\cap\nap$, cf{.} the proof of Proposition~\ref{prop:1}. Hence our task is to prove that if
$k=\# \eus K(\n)\ge 2$ and $\eus K(\n)$ is not contained in an interval $[\beta,\beta_1]_{\eus K}$ 
for some $\beta\in\Phi^{-1}(\Pi_{\sf com})$, then different abelian pieces $\{\ah_i\}_{i=1}^k$ cannot 
compatibly be matched.

To this end, we use a simple counting argument. Clearly, if $\ah\lhd\n$ is an abelian $\be$-ideal, then 
$\Delta(\ah)\subset \Delta^+_{\sf com}$. Moreover, for each $i$, the roots in 
$\eus H_{\beta_i}\cap \Delta(\ah)=\Delta(\ah_i)$ are commutative in $\Delta^+$. Therefore, if $\ah$ is 
a \CP-ideal in an optimal nilradical $\n$, then 
\beq        \label{eq:admissible}
    \# (\eus H_{\beta_i}\cap \Delta^+_{\sf com})\ge (\dim\h_i +1)/2=(\#\eus H_{\beta_i}+1)/2 
\eeq
for $i=1,\dots,k$. Therefore, if $\n$ has a \CP\ and $\beta_j\in\eus K$ does not 
satisfy~\eqref{eq:admissible}, then $\beta_j$ cannot occur in $\eus K(\n)$.
Let us say that $\beta_j\in\eus K$ is {\it admissible}, if~\eqref{eq:admissible}
holds for $\beta_j$. Note that if $\beta_j$ is admissible, then $\beta_j\in\Delta^+_{\sf com}$.
An upper ideal $\eus J\subset \eus K$ is said to be {\it admissible}, if {\bf all} elements of
$\eus J$ are admissible. If $\eus J_1$ and $\eus J_2$ are admissible, then so is 
$\eus J_1\cup\eus J_2$. Therefore, there is a unique maximal admissible upper ideal 
$\eus J_{\sf adm}\subset \eus K$. It also follows from the preceding argument that if $\n$ has a \CP, 
then $\eus K(\n)$ is admissible and thereby $\eus K(\n)\subset \eus J_{\sf adm}$. 

Since the elements of $\eus K$ and $\Delta^+_{\sf com}$ are known, it is not hard to determine 
$\eus J_{\sf adm}$ for all simple $\g$. If $\eus J_{\sf adm}$ appears to be a good chain, then this
immediately proves Theorem~\ref{thm:CP-shtrikh} for the corresponding $\g$. If $\eus J_{\sf adm}$ is 
not a sole chain in $\eus K$, then some extra argument is needed. In our case-by-case proof below, 
we use the numbering of $\{\beta_i\}$ given in Section~\ref{sect:tables}. 
The interested reader should always consult Figures~\ref{fig:An}--\ref{fig:En}, 
where the Hasse diagrams of the {\sf MCP} $(\eus K,\curle, \Phi)$ are depicted.

\begin{proof}[Proof of Theorem~\ref{thm:CP-shtrikh}]
\leavevmode\par
{\it\bfseries I.} \ For $\GR{A}{n}$ and $\GR{C}{n}$, all nilradicals have a \CP. 
Therefore, we omit these series.
Actually, it follows from Remark~\ref{rem:recover} that here $\eus J_{\sf adm}=\eus K$, which is a good chain in both cases.

{\it\bfseries II.}  \ Let $\g$ be exceptional. Then all commutative roots can be found via the use of 
\eqref{eq:non-com}, cf. also~\cite[Appendix\,A]{p06}.
\begin{itemize}
\item 
For $\GR{G}{2}$, $\GR{F}{4}$, and $\GR{E}{8}$, the highest root $\beta_1=\theta$ has the unique 
descendant $\beta_2$. For $\GR{F}{4}$ and $\GR{E}{8}$, one has $\beta_2\in\Delta^+_{\sf com}$.
But in all three cases
$(\#\eus H_{\beta_2}+1)/2 > \# (\eus H_{\beta_2}\cap \Delta^+_{\sf com})$. The relevant numbers
$\bigl(\#\eus H_{\beta_2}, \# (\eus H_{\beta_2}\cap \Delta^+_{\sf com})\bigr)$ are $(1,0),\,(5,1),\,(33,8)$,  
respectively.
Hence $\eus J_{\sf adm}=\{\beta_1\}$ and $\h_1=\n_\theta$ is the only nilradical having a \CP.

\item For $\GR{E}{7}$, one has $\beta_4\curle [\theta/2]$, hence $\beta_4\not\in\Delta^+_{\sf com}$,
cf.~\eqref{eq:non-com}. Then
$\eus J_{\sf adm}=[\beta_3,\beta_1]_{\eus K}$, which is a good chain, see Figure~\ref{fig:En}.

\item For $\GR{E}{6}$, one similarly has $\beta_3\not\in\Delta^+_{\sf com}$. Hence
$\eus J_{\sf adm}=[\beta_2,\beta_1]_{\eus K}$, which is a good chain.
\end{itemize}
This proves Theorem~\ref{thm:CP-shtrikh} for the exceptional Lie algebras.

{\it\bfseries III.}  \ Let $\g=\son$ be an orthogonal Lie algebra. We use below the explicit description of 
$\Delta^+_{\sf com}$ via $\{\esi_i\}$'s, which is given in the proof of Theorem\,4.4 in~\cite{p06}.
\begin{itemize}
\item For $\GR{B}{n}$ with $n\ge 4$, one has $\beta_3=\esi_3+\esi_4$ and 
\[
   \eus H_{\beta_3}=\{\esi_3\pm\esi_j \mid j=5,\dots,n\}\cup\{\esi_4\pm\esi_j \mid j=5,\dots,n\}
   \cup\{\esi_3,\esi_4, \beta_3\} .
\]
Hence $\# \eus H_{\beta_3}=4(n-4)+3=4n-13$. On the other hand,
\[
   \eus H_{\beta_3}\cap\Delta^+_{\sf com}=\{\esi_3+\esi_j \mid j=5,\dots,n\}\cup\{\esi_4+\esi_j \mid j=5,\dots,n\}\cup\{\beta_3\} ,
\]
and $\# (\eus H_{\beta_3}\cap \Delta^+_{\sf com})=2(n-4)+1=2n-7$. Therefore, \eqref{eq:admissible}
does not hold for $\beta_3$  
and  $\eus J_{\sf adm}=\{\beta_1,\beta_2\}=[\beta_2,\beta_1]_\eus K$, which 
is a good chain (see Fig.~\ref{fig:Bn}).

\item For $\GR{B}{3}$, one has $\eus H_{\beta_3}=\{\beta_3\}=\{\esi_3\}$ and $\esi_3\not\in
\Delta^+_{\sf com}$. Here again $\eus J_{\sf adm}=\{\beta_1,\beta_2\}$, see Figure~\ref{fig:D4,B3,G2}.
This proves Theorem~\ref{thm:CP-shtrikh} for series $\GR{B}{n}$.

\item For $\GR{D}{n}$ with $n\ge 4$, we have $\eus H_{\beta_{2k}}=\{\beta_{2k}\}$ and
$\beta_{2k}=\esi_{2k-1}-\esi_{2k}=\ap_{2k-1}\not\in \Delta^+_{\sf com}$ for $4\le 2k< n$. On the hand,
condition~\eqref{eq:admissible} holds for all other $\beta_i\in\eus K$.
Therefore, the minimal elements of $\eus J_{\sf adm}$ are 
$\left\{\begin{array}{rcl} \beta_{2}, \beta_{2n-1}, \beta_{2n} & , & \text{ for } \  \g=\GR{D}{2n} \\
\beta_{2}, \beta_{2n-1} & , & \text{ for } \  \g=\GR{D}{2n+1} \end{array}\right.$ and $\eus J_{\sf adm}$ is 
a union of several good chains (three for $\GR{D}{2n}$ and two for $\GR{D}{2n+1}$), see 
Figure~\ref{fig:Dn}. Hence an additional argument is needed, which is supplied below.
\end{itemize} 

\begin{prop}    
\label{prop:extra-arg}
Let $\g$ be of type $\GR{D}{n}$ ($n\ge 4$) and\/ $\n$ an optimal nilradical. 
\begin{enumerate}
\item If \ $\beta_2,\beta_3\in\eus K(\n)$, then $\n$ has no \CP;
\item if \ $\beta_{2n-1},\beta_{2n}\in\eus K(\n)$ for\/ $\GR{D}{2n}$, then $\n$ has no \CP.
\end{enumerate}
\end{prop}
\begin{proof}
(1) Assume that $\ah\lhd\n$ is a \CP-ideal. Since $\eus H_{\beta_2}=\{\beta_2\}=\{\ap_1\}$, $\Delta(\ah)$ 
must contain $\ap_1=\esi_1-\esi_2$. Hence $\ah$ contains the abelian ideal $\ce\lg\ap_1\rg$, i.e., 
$\Delta(\ah)$ contains $I\lg\ap_1\rg=\{ \esi_1\pm\esi_j\mid j=2,\dots, n\}$, which is an upper ideal of 
$\Delta^+$. Here $I\lg\ap_1\rg\setminus\{\ap_1\}\subset\eus H_{\beta_1}$. 
Since $\dim(\ah\cap\h_3)=(\dim\h_3+1)/2$, we must also add some roots from $\eus H_{\beta_3}$ to
$I\lg\ap_1\rg$. But it is easily seen that $\ce\lg\ap_1\rg$ is a maximal abelian ideal of $\be$, hence no 
roots from $\eus H_{\beta_3}$ can be added to $I\lg\ap_1\rg$. A contradiction!

(2) Here $\beta_{2n-1}=\ap_{2n-1}=\esi_{2n-1}+\esi_{2n}$ and 
$\beta_{2n}=\ap_{2n}=\esi_{2n-1}-\esi_{2n}$. If $\n$ is optimal and 
$\beta_{2n-1},\beta_{2n}\in\eus K(\n)$, then the Hasse diagram of $\eus K$ for type $\GR{D}{2n}$ (Fig.~\ref{fig:Dn}) shows that $\beta_3\in\eus K(\n)$ and hence $\ap_1=\beta_2\not\in \eus K(\n)$.
Then $\n=\n_T$, where 
$T=\Pi\setminus \{\ap_1,\ap_3,\dots,\ap_{2n-3}\}$. Therefore,
\[
   \dim\n=\dim\ut-(n-1)=4n^2-3n+1, \quad \ind\n=\#T=n+1  .
\] 
Hence $\bb(\n)=2n^2-n+1$. On the other hand, for $\g$ of type $\GR{D}{2n}$, the maximal dimension 
of an abelian ideal of $\be$ equals $\genfrac{(}{)}{0pt}{}{2n}{2}=2n^2-n$, see e.g.~\cite[Section\,9]{cp3}.
\end{proof}

It follows from Proposition~\ref{prop:extra-arg} that if $\n$ is optimal and has a \CP, then
$\eus K(\n)$ is contained in a good chain whose minimal element is:

{\bf --} \  $\beta_2$ or $\beta_{2n-1}$ or  $\beta_{2n}$ for $\GR{D}{2n}$; 

{\bf --} \  $\beta_2$ or $\beta_{2n-1}$ for $\GR{D}{2n+1}$.
\\
This completes our case-by-case verification if $\g$ is of type $\GR{D}{n}$. Thus,
Theorem~\ref{thm:CP-shtrikh} is proved.  
\end{proof}

Theorems~\ref{thm:CP-da} and \ref{thm:CP-shtrikh} provide a complete characterisation of the nilradicals
with \CP.
\begin{rmk}    \label{rem:by-prod}
It follows from our classification that, for a nilradical $\n$,  the property of having \CP\ depends only 
on $\eus K(\n)$. In particular, $\n$ has a \CP\ if and only if $\tilde\n$ does. Hence the 
converse of Lemma~\ref{lm:optim2} is also true, and it might be useful to find a direct proof for it.
However, this equivalence does not mean that every \CP-ideal for $\tilde\n$ is necessarily a \CP-ideal for $\n$, see Example~\ref{ex:sln}(2).
 
\end{rmk}

\begin{ex}        \label{ex:small-rk} 
Using two our theorems and Hasse diagrams from Section~\ref{sect:tables}, we give
the complete list of nilradicals with \CP\ for some small rank cases.
First, we assume that $\n=\n_T$ is optimal and point out the possible subsets $T\subset\Pi$.
\begin{description}
\item[$\GR{D}{4}$] \  $\{\ap_2\}$, $\{\ap_2,\ap_1\}$, $\{\ap_2,\ap_3\}$, $\{\ap_2,\ap_4\}$;
\item[$\GR{D}{5}$] \  $\{\ap_2\}$, $\{\ap_2,\ap_1\}$, $\{\ap_2,\ap_4,\ap_5\}$;
\item[$\GR{E}{6}$] \  $\{\ap_6\}$, $\{\ap_6,\ap_1, \ap_5\}$;
\item[$\GR{E}{7}$] \  $\{\ap_6\}$, $\{\ap_6,\ap_2\}$, $\{\ap_6,\ap_2, \ap_1\}$.
\end{description}
Then an arbitrary nilradical with \CP\ corresponds to a subset of any of the above sets.
\end{ex}

\begin{ex}[Continuation of \ref{ex:sl-2}]    \label{ex:sl-contin}
For $T=\{\ap_{i_1}, \dots, \ap_{i_k} \}$ and $\n_T\subset\slno$,  the index $\bi$ is defined by 
Eq.~\eqref{eq:nomer-i}. If both $\ap_\bi$ and $\ap_{n+1-\bi}$ belong to $T$, then $\n_T\cap\n_{\{\ap_j\}}$ 
is a \CP-ideal of $\n_T$ for each $j \in [\bi,\dots, n+1-\bi]$. But if only one of them, say $\ap_\bi$, belongs to $T$, then $\n_{\{\ap_\bi\}}$ is the only \CP-ideal in $\n_T$.
\\ \indent
For instance, if $T=\{\ap_1,\ap_n\}$, then $\n_T=\h_1$ is the Heisenberg nilradical and $\bi=1$. Here 
$\h_1\cap\nap$ is a \CP-ideal for {\bf any} $\ap\in\Pi$, which is a manifestation of 
Proposition~\ref{prop:heis} for $\slno$.
\end{ex}

\section{Some complements}
\label{sect:compl}

In this section, we obtain some extra results related to the optimisation and optimal nilradicals. We use 
capital Latin letters for connected subgroups of $G$ corresponding to algebraic Lie subalgebras of $\g$.
For instance, the chain of Lie algebras $\n\subset\be\subset\p$ gives rise to the chain of groups $N\subset B\subset P$.

Let $\n=\p^{\sf nil}$ be an arbitrary standard nilradical and $\n^-=(\p^-)^{\sf nil}$  the opposite nilradical, 
i.e., $\Delta(\n^-)=-\Delta(\n)$. Using the direct sum of vector spaces $\g=\p\oplus\n^-$ and the 
$P$-module isomorphism $\n^*\simeq \g/\p$, we identify $\n^*$ with $\n^-$ and thereby regard $\n^-$ as 
$P$-module. In terms of $\n^-$, the $\p$-action on $\n^*$ is given by the Lie bracket in $\g$, with the 
subsequent projection to $\n^-$. In particular, the coadjoint representation of $\n$ has the following 
interpretation. If $x\in\n$ and $\xi\in\n^-$, then $(\ads_\n x){\cdot}\xi=\mathsf{pr}_{\n^-} ([x,\xi])$, where 
$\mathsf{pr}_{\n^-}: \g\to\n^-$ is the projection with kernel $\p$. The same principle applies below to Lie 
algebras $\q$ such that $\n\subset\q\subset\p$.

Recall that the set of {\it regular\/} elements of $\n^*$ (w.r.t. $\ads_\n$), denoted $\n^*_{\sf reg}$, 
consists of all $\xi$ such that the coadjoint orbit 
$N{\cdot}\xi:=\Ad^*_N(N){\cdot}\xi\subset\n^*$ has the maximal dimension. 
Set $\bxi=\sum_{\beta\in\eus K(\n)}e_{-\beta}\in \n^-\simeq \n^*$. By~\cite[2.4]{jos77}, if $\bxi$ is 
regarded as element of $\n^*$, then $\bxi\in\n^*_{\sf reg}$ and $B{\cdot}\bxi$ is the dense $B$-orbit 
in $\n^*$. Furthermore, if $\n$ is optimal, then  
the stabiliser of $\bxi$ in $\n$, $\n^{\bxi}$, equals \ $\bigoplus_{\beta\in\eus K(\n)}\g_\beta$, which also 
shows that  $\ind\n=\#\eus K(\n)$.
We say that $\bxi$ is a {\it cascade point\/} in $\n^-$. 

For a nilradical $\n$, consider the subalgebra
$\te_{\n}=\bigoplus_{\beta\in\eus K(\n)} [\g_\beta,\g_{-\beta}]\subset\te$ and the
solvable Lie algebra $\ff_{\n}=\tilde\n\oplus \te_{\n}\subset\be$. Note that 
$\te_\n=\te_{\tilde\n}$, $\ff_\n=\ff_{\tilde\n}$, and $\ff_\n$ is a non-abelian ideal of $\be$. 

\begin{prop}    \label{prop:c-Frob}
For any nilradical\/ $\n$, the Lie algebra $\ff_\n$ is Frobenius, i.e., $\ind\ff_\n=0$.
\end{prop}
\begin{proof}   We write $\ff$ for $\ff_\n$ in the proof. 
Since $\tilde\n$ is an ideal of $\ff$, the group $F\subset B$ acts on both $\tilde\n$ and $\tilde\n^*$. Let
$\ind(\ff,\tilde\n)$ denote the minimal codimension of the $F$-orbits in $\tilde\n^*$.
In~\cite[Theorem\,1.4]{cambr}, an upper bound is given on the sum of indices of a Lie algebra and an
ideal in it. For the pair $(\ff,\tilde\n)$, it reads
\[
  \ind\ff+\ind\tilde\n \le 2\,\ind(\ff,\tilde\n) +\dim(\ff/\tilde\n) .
\]
In our case, $\ind\tilde\n=\dim(\ff/\tilde\n)$ (Corollary~\ref{cor:jos}). Hence it suffices to prove that
$\ind(\ff,\tilde\n)=0$, i.e., $F$ has a dense orbit in $\tilde\n^*$. To this end, we use the cascade
point $\bxi\in\tilde\n^*$. Since $\tilde\n^{\bxi}=\bigoplus_{\gamma\in\eus K(\n)}\g_\gamma$, we have 
$\tilde\n{\cdot}\bxi=(\tilde\n^{\bxi})^\perp\simeq\bigoplus_\gamma \g_{-\gamma}$, where $\gamma$ 
runs over $\Delta(\tilde\n)\setminus \eus K(\n)$. On the other hand, the roots in $\eus K(\n)$ are 
strongly orthogonal. Hence $\te_\n{\cdot}\bxi=\bigoplus_{\gamma\in\eus K(\n)} \g_{-\gamma}$ and 
thereby $\ff{\cdot}\bxi=(\tilde\n\oplus\te_\n){\cdot}\bxi=\tilde\n^-$.
\end{proof}

{\bf Remark.} Considering $\bxi$ as element of $\ff^*$, one can prove directly that the stabiliser of ${\bxi}$ in $\ff$  is trivial, hence $\ff{\cdot}\bxi=\ff^*$ and the orbit $F{\cdot}\bxi$ is dense in $\ff^*$.

The Lie algebra $\ff_\n$ is said to be the {\it Frobenius envelope\/} of $\n$. Note that the nilradical of 
$\ff_\n$ is $\tilde\n$. For  $\n=\ut$, the Frobenius envelope $\ff_\ut\subset \be$ is the Lie subalgebra of 
dimension $\dim\ut+\#\eus K$. Hence \ \ 
{$\ff_\ut=\be$ \ {\sl if and only\/} \ if $\#\eus K=\#\Pi$ \ {\sl if and only if\/} \ $\be$ is 
Frobenius.}

\begin{prop}      \label{prop:bb=ravno}
For any $\n$, one has\/ $\bb(\ff_\n)=\bb(\n)$ and\/ $\ff_\n$ has a \CP\ if and only if\/ $\n$ does.
\end{prop}
\begin{proof} 
Since $\dim\ff_\n=\dim\n+\ind\n$~\cite[Prop.\,2.6]{jos77} and $\ind\ff_\n=0$, we obtain the first assertion.

If $\ah\lhd\ff_\n$ is a \CP-ideal, then it is an abelian $\be$-ideal, hence $\ah\subset\ut$. Then 
$\ah\subset\ut\cap\ff_n=\tilde\n$ and $\ah$ is a \CP-ideal in $\tilde\n$. Then $\n$ also has a \CP, see 
Remark~\ref{rem:by-prod}.
\end{proof}

For an arbitrary nilradical $\n$, Joseph's formula $\ind\n=\dim\ff_\n-\dim\n$ shows that $\ind\n$ can be 
considerably larger than $\rk\g$. On the other hand, if $\n$ is optimal, then one always has 
$\ind\n=\#\eus K(\n)\le\rk\g$. In fact, there is a more informative assertion.

\begin{thm}                   \label{thm:summa-opt}
If\/ $\p$ is an optimal parabolic subalgebra and\/ $\n=\p^{\sf nil}$, then $\ind\p+\ind\n=\rk\g$. 
\end{thm}
\begin{proof}
Because $\ind\n=\#\eus K(\n)$, we have to prove that $\ind\p=\rk\g-\#\eus K(\n)$. For simplicity, we again 
write $\ff$ for the Frobenius envelope of the optimal nilradical $\n$.
We have 
\centerline{$\p=\el\oplus\n=\es\oplus\te_\n\oplus\n=\es\oplus\ff$,}
\\[.6ex] 
where $\es$ is a reductive Lie algebra such that $[\el,\el]\subset\es \subset\el$ and 
$\es\cap\te=(\te_\n)^\perp$. Therefore $\rk\es=\rk\g-\#\eus K(\n)$. 
Accordingly, $\p^-=\es\oplus\ff^-$. Since $\p^*\simeq \g/\n$ as $\p$-module, we identify $\p^*$ with 
$\p^-$ and regard the cascade point $\bxi=\sum_{\beta\in\eus K(\n)}e_{-\beta}$ as element of  
$\ff^-\subset\p^-\simeq\p^*$. 

Let $\ads_\p: \p\to\mathfrak{gl}(\p^-)$ (resp. $\Ad^*_P: P\to GL(\p^-)$) denote the coadjoint 
representation of $\p$ (resp. $P$). Since $\ff\lhd\p$ and $\p^-=\p^*\to \ff^*=\ff^-$ is a surjective
homomorphism of $F$-modules, one easily verifies for any $g\in F$ that
\begin{itemize}
\item[\sf (1)] \  $\Ad^*_P(g){\cdot}\eta=\eta$ for any $\eta\in\es$;
\item[\sf (2)] \  $\Ad^*_P(g){\cdot}\xi-\Ad^*_F(g){\cdot}\xi\in\es$  for any $\xi\in\ff^-$. 
In other words, $\Ad^*_P(g){\cdot}\xi=\Ad^*_F(g){\cdot}\xi+\tilde\eta$ for some $\tilde\eta\in\es$.
\end{itemize}

\noindent
Since $\ff$ is Frobenius (Prop.~\ref{prop:c-Frob}), $F$ has the open orbit in $\ff^*$.
If $\Omega\subset\ff^*$ is this $F$-orbit, then we identify $\Omega$ with the open subset of $\ff^-$
that contains $\bxi$.
Let $\mu=\eta+\xi\in\p^-$ be a $P$-generic point, where $\eta\in\es$ and $\xi\in\ff^-$. Without loss of 
generality, we may assume that $\xi\in\Omega$. It follows from {\sf (2)} above that there is $g_0\in F$ 
such that $\Ad^*_P(g_0){\cdot}\xi=\bxi+\tilde\eta$. Hence 
$\Ad^*_P(g_0){\cdot}(\xi+\eta)=\bxi+\eta+\tilde\eta=\bxi+\eta'$. Therefore, a $P$-generic point in $\p^-$ 
can be taken as a sum of $\bxi$ and an $S$-generic point $\eta'\in\es$.

Our next goal is to compute the dimension of 
$\ads_\p(\p){\cdot}(\eta'+\bxi)=:\p{\cdot}(\eta'+\bxi)$. We have
\[
   \p{\cdot}(\eta'+\bxi)=(\es+\ff){\cdot}(\eta'+\bxi)=\es{\cdot}(\eta'+\bxi)+\ff{\cdot}(\eta'+\bxi) .
\]
Note that $[\ff,\es]\subset\n$, which it is zero in $\p^-=\p^*$. Hence $\ff{\cdot}\eta'=0$ for any $\eta'\in\es$. 
On the other hand, the construction of $\es$ via
$\eus K(\n)$ and properties of $\eus K(\n)$ show that  $[\es,\bxi]=0$.

Therefore, $\p{\cdot}(\eta'+\bxi)=\es{\cdot}\eta'+\ff{\cdot}\bxi$.
Here $\es{\cdot}\eta'\subset \es\subset\p^-$ and $\dim\es{\cdot}\eta'=\dim\es-\rk\es$. The elements of 
$\ff{\cdot}\bxi$ can be written  as
$\ads_\p (x){\cdot}\bxi=\xi_1+\xi_2$, $x\in\ff$, where $\xi_1\in\ff^-$ and $\xi_2\in\es$. The Lie algebra analogue of {\sf (2)} shows that $\xi_1=\ad_\ff^*(x){\cdot}\bxi$. Since $\ff$ is Frobenius and
$\bxi$ is an element of the open $F$-orbit in $\ff^*$, we have $\xi_1\ne 0$ for any $x\ne 0$. Hence
$\es{\cdot}\eta'\cap\ff{\cdot}\bxi=\{0\}$ and thereby 
\[
    \dim(\es{\cdot}\eta'+\ff{\cdot}\bxi)=(\dim\es -\rk\es)+\dim\ff=\dim\p - (\rk\g- \#\eus K(\n)) .
\]
Thus, $\ind\p= \rk\g -\#\eus K(\n)$, as required.
\end{proof}

\begin{rmk}    \label{rem:history}
(1) The fact that $\ind\be+\ind\ut=\rk\g$ is mentioned in \cite[Remark~1.5.1]{cambr} together with some hints 
on a proof. It is also proved in~\cite[Corollary\,1.5]{cambr} that $\ind\p+\ind\p^{\sf nil}\le \dim\el$, where 
$\el$ is a Levi subalgebra of $\p$. On the other hand, I conjectured in \cite[Section\,6]{cambr} that
$\ind\p+\ind\p^{\sf nil}\ge\rk\g$ for any parabolic $\p$. This has been proved afterwards by 
R.~Yu~\cite{yu} via the use of a general formula for the index of a parabolic subalgebra. Yu also points 
out the cases, where $\ind\p+\ind\p^{\sf nil}=\rk\g$. However, the special case of the optimal parabolic subalgebras, where the proof is simpler and shorter, deserves to be presented, too.

(2) One can similarly prove that if $\p$ is an {\sl optimal\/} parabolic subalgebra and $\rr:=\es\oplus\n$, then $\ind\rr=\rk\g$. Therefore, associated with such $\p$, we obtain two vector
space decompositions of $\g$ with an interesting property:
\[ 
  \g=\p\oplus\n^-  \quad \text{and} \quad 
  \g=\rr\oplus \ff^-_{\n} .
\] 
In both cases, the sum of indices of summands equals the index of $\g$ ($=\rk\g$).
\end{rmk}

We have already encountered several instances, where $\bb(\q')=\bb(\q)$ for different Lie algebras
$\q'\subset \q$ (see Corollary~\ref{cor:jos} and Proposition~\ref{prop:bb=ravno}). 
We conclude this section with a simple observation that may have interesting
invariant-theoretic applications. Recall that the algebra of symmetric invariants, $\gS(\q)^\q$, is also the Poisson centre of $\gS(\q)$.

\begin{prop}     \label{prop:inv-th}
\leavevmode\par
\begin{itemize}
\item[\sf (i)] \ If\/ $\q'\subset\q$ \ are Lie algebras and $\bb(\q')=\bb(\q)$, then $\gS(\q)^\q\subset \gS(\q')$.
\item[\sf (ii)] \ In particular, if\/ $\ah\subset\q$ is a \CP, then $\gS(\q)^\q\subset \gS(\ah)$.
\item[\sf (iii)] \ Moreover, if\/ $\ah_1,\dots,\ah_k$ are different \CP\ in $\q$, then
$\gS(\q)^\q\subset \gS(\bigcap_{i=1}^k\ah_i)$.
\end{itemize}
\end{prop}
\begin{proof}
{\sf (i)} \ By Sadetov's theorem, there is a Poisson-commutative subalgebra $\ca'\subset\gS(\q')$
such that $\trdeg\ca'=\bb(\q')$~\cite{sad}, see also Introduction. By the assumption, $\ca'$ is
also a Poisson-commutative subalgebra of maximal transcendence degree in $\gS(\q)$. The 
subalgebra of $\gS(\q)$ generated by $\ca'$ and $\gS(\q)^\q$, say $\ca$, is still Poisson-commutative
and if $\gS(\q)^\q\not\subset \gS(\q')$, then $\trdeg\ca>\trdeg\ca'$.

{\sf (ii)} \ If $\ah$ is a \CP\ in $\q$, then $\bb(\ah)=\dim\ah=\bb(\q)$.

{\sf (iii)} \ This follows from the fact that \ $\bigcap_{i=1}^k \gS(\ah_i)=\gS(\bigcap_{i=1}^k\ah_i)$.
\end{proof}

In a forthcoming publication, we provide some applications of this result to symmetric invariants of
the nilradicals of parabolic subalgebras in $\g$.

\section{The elements of $\eus K$ and Hasse diagrams}
\label{sect:tables}

\noindent
In this section, we provide the lists of cascade elements and the Hasse diagrams of cascade posets 
$\eus K$ for all simple Lie algebras, see Fig.~\ref{fig:An}--\ref{fig:En}. To each node $\beta_j$ in the 
Hasse diagram, the set $\Phi(\beta_j)\subset\Pi$ is attached. The main features are:
\begin{itemize}
\item 
$\Pi=\{\ap_1,\dots,\ap_{\rk \g}\}$ and the numbering of $\Pi$ 
follows~\cite[Table\,1]{VO}. 
\item The commutative simple roots in the sets $\Phi(\beta_j)$, $j=1,\dots,m$,  are underlined.
\item The numbering of the $\beta_i$'s  in the lists corresponds to that in the figures.
\end{itemize}
We use below the standard notation for roots of the classical Lie algebras, see~\cite[Table\,1]{VO}.

\noindent
{\it\bfseries The list of cascade elements for the classical Lie algebras}:
\begin{description}
\item[$\GR{A}{n}, n\ge 2$]  \ $\beta_i=\esi_i-\esi_{n+2-i}=\ap_i+\dots +\ap_{n+1-i}$ \ ($i=1,2,\dots,\left[\frac{n+1}{2}\right]$);
\item[$\GR{C}{n}, n\ge 1$]  \ $\beta_i=2\esi_i=2(\ap_i+\dots+\ap_{n-1})+\ap_n$ \ ($i=1,2,\dots,n-1$) and 
$\beta_n=2\esi_n=\ap_n$;
\item[$\GR{B}{2n}$, $\GR{D}{2n}$, $\GR{D}{2n+1}$ ($n\ge 2$)]  \ 
$\beta_{2i-1}=\esi_{2i-1}+\esi_{2i}$, $\beta_{2i}=\esi_{2i-1}-\esi_{2i}$ \ ($i=1,2,\dots,n$);
\item[$\GR{B}{2n+1}, n\ge 1$] \ here $\beta_1,\dots,\beta_{2n}$ are as above and $\beta_{2n+1}=\esi_{2n+1}$;
\end{description}

\noindent
For all orthogonal series, we have $\beta_{2i}=\ap_{2i-1}$, $i=1,\dots,n$, while formulae for 
$\beta_{2i-1}$ via $\Pi$ slightly differ for different series. E.g. for $\GR{D}{2n}$ one has
$\beta_{2i-1}=\ap_{2i-1}+2(\ap_{2i}+\dots+\ap_{2n-2})+\ap_{2n-1}+\ap_{2n}$ ($i=1,2,\dots,n-1$)
and $\beta_{2n-1}=\ap_{2n}$.

\noindent
{\it\bfseries The list of cascade elements for the exceptional Lie algebras}:
\begin{description}
\item[$\GR{G}{2}$]  \ $\beta_1=(32)=3\ap_1+2\ap_2, \ \beta_2=(10)=\ap_1$;
\item[$\GR{F}{4}$]   \ $\beta_1=(2432)=2\ap_1+4\ap_2+3\ap_3+2\ap_4,\ \beta_2=(2210),\ \beta_3=(0210),\ \beta_4=(0010)$;
     \item[$\GR{E}{6}$] \  
  $\beta_1=$\raisebox{-2.1ex}{\begin{tikzpicture}[scale= .85, transform shape]
\node (a) at (0,0) {1}; \node (b) at (.2,0) {2};
\node (c) at (.4,0) {3}; \node (d) at (.6,0) {2};
\node (e) at (.8,0) {1}; \node (f) at (.4,-.4) {2};
\end{tikzpicture}}\!\!, 
  $\beta_2=$\raisebox{-2.1ex}{\begin{tikzpicture}[scale= .85, transform shape]
\node (a) at (0,0) {1}; \node (b) at (.2,0) {1};
\node (c) at (.4,0) {1}; \node (d) at (.6,0) {1};
\node (e) at (.8,0) {1}; \node (f) at (.4,-.4) {0};
\end{tikzpicture}}, 
  $\beta_3=$\raisebox{-2.1ex}{\begin{tikzpicture}[scale= .85, transform shape]
\node (a) at (0,0) {0}; \node (b) at (.2,0) {1};
\node (c) at (.4,0) {1}; \node (d) at (.6,0) {1};
\node (e) at (.8,0) {0}; \node (f) at (.4,-.4) {0};
\end{tikzpicture}},
  $\beta_4$\,=\raisebox{-2.1ex}{\begin{tikzpicture}[scale= .85, transform shape]
\node (a) at (0,0) {0}; \node (b) at (.2,0) {0};
\node (c) at (.4,0) {1}; \node (d) at (.6,0) {0};
\node (e) at (.8,0) {0}; \node (f) at (.4,-.4) {0};
\end{tikzpicture}}=\,$\ap_3$;
     \item[$\GR{E}{7}$] \ 
  $\beta_1$=\raisebox{-2.1ex}{\begin{tikzpicture}[scale= .85, transform shape]
\node (a) at (0,0) {1}; \node (b) at (.2,0) {2}; \node (c) at (.4,0) {3};
\node (d) at (.6,0) {4}; \node (e) at (.8,0) {3}; \node (f) at (1,0) {2};
\node (g) at (.6,-.4) {2};
\end{tikzpicture}},
  $\beta_2$=\raisebox{-2.1ex}{\begin{tikzpicture}[scale= .85, transform shape]
\node (a) at (0,0) {1}; \node (b) at (.2,0) {2}; \node (c) at (.4,0) {2};
\node (d) at (.6,0) {2}; \node (e) at (.8,0) {1}; \node (f) at (1,0) {0};
\node  at (.6,-.4) {1};
\end{tikzpicture}},
  $\beta_3$=\raisebox{-2.1ex}{\begin{tikzpicture}[scale= .85, transform shape]
\node (a) at (0,0) {1}; \node (b) at (.2,0) {0}; \node (c) at (.4,0) {0};
\node (d) at (.6,0) {0}; \node (e) at (.8,0) {0}; \node (f) at (1,0) {0};
\node (g) at (.6,-.4) {0};
\end{tikzpicture}}=\,$\ap_1$,
  $\beta_4$=\raisebox{-2.1ex}{\begin{tikzpicture}[scale= .85, transform shape]
\node (a) at (0,0) {0}; \node (b) at (.2,0) {0}; \node (c) at (.4,0) {1};
\node (d) at (.6,0) {2}; \node (e) at (.8,0) {1}; \node (f) at (1,0) {0};
\node (g) at (.6,-.4) {1};
\end{tikzpicture}},
  $\beta_5$=\raisebox{-2.1ex}{\begin{tikzpicture}[scale= .85, transform shape]
\node (a) at (0,0) {0}; \node (b) at (.2,0) {0}; \node (c) at (.4,0) {1};
\node (d) at (.6,0) {0}; \node (e) at (.8,0) {0}; \node (f) at (1,0) {0};
\node (g) at (.6,-.4) {0};
\end{tikzpicture}}=\,$\ap_3$,   \\
  $\beta_6$=\raisebox{-2.1ex}{\begin{tikzpicture}[scale= .85, transform shape]
\node (a) at (0,0) {0}; \node (b) at (.2,0) {0}; \node (c) at (.4,0) {0};
\node (d) at (.6,0) {0}; \node (e) at (.8,0) {1}; \node (f) at (1,0) {0};
\node (g) at (.6,-.4) {0};
\end{tikzpicture}}=\,$\ap_5$,
  $\beta_7$=\raisebox{-2.1ex}{\begin{tikzpicture}[scale= .85, transform shape]
\node (a) at (0,0) {0}; \node (b) at (.2,0) {0}; \node (c) at (.4,0) {0};
\node (d) at (.6,0) {0}; \node (e) at (.8,0) {0}; \node (f) at (1,0) {0};
\node (g) at (.6,-.4) {1};
\end{tikzpicture}}=\,$\ap_7$;
      \item[$\GR{E}{8}$] \ 
  $\beta_1$=\raisebox{-2.1ex}{\begin{tikzpicture}[scale= .85, transform shape]
\node (a) at (0,0) {2}; \node (b) at (.2,0) {3}; \node (c) at (.4,0) {4};
\node (d) at (.6,0) {5}; \node (e) at (.8,0) {6}; \node (f) at (1,0) {4}; \node (g) at (1.2,0) {2};
\node (h) at (.8,-.4) {3};
\end{tikzpicture}},
  $\beta_2$=\raisebox{-2.1ex}{\begin{tikzpicture}[scale= .85, transform shape]
\node (h) at (-.2,-.0) {0}; \node (a) at (0,0) {1}; \node (b) at (.2,0) {2}; \node (c) at (.4,0) {3};
\node (d) at (.6,0) {4}; \node (e) at (.8,0) {3}; \node (f) at (1,0) {2};
\node (g) at (.6,-.4) {2};
\end{tikzpicture}},
  $\beta_3$=\raisebox{-2.1ex}{\begin{tikzpicture}[scale= .85, transform shape]
\node (h) at (-.2,-.0) {0}; \node (a) at (0,0) {1}; \node (b) at (.2,0) {2}; \node (c) at (.4,0) {2};
\node (d) at (.6,0) {2}; \node (e) at (.8,0) {1}; \node (f) at (1,0) {0};
\node  at (.6,-.4) {1};
\end{tikzpicture}},
  $\beta_4$=\raisebox{-2.1ex}{\begin{tikzpicture}[scale= .85, transform shape]
\node (h) at (-.2,-.0) {0}; \node (a) at (0,0) {1}; \node (b) at (.2,0) {0}; \node (c) at (.4,0) {0};
\node (d) at (.6,0) {0}; \node (e) at (.8,0) {0}; \node (f) at (1,0) {0};
\node (g) at (.6,-.4) {0};
\end{tikzpicture}}=\,$\ap_2$,
  $\beta_5$=\raisebox{-2.1ex}{\begin{tikzpicture}[scale= .85, transform shape]
\node (h) at (-.2,-.0) {0}; \node (a) at (0,0) {0}; \node (b) at (.2,0) {0}; \node (c) at (.4,0) {1};
\node (d) at (.6,0) {2}; \node (e) at (.8,0) {1}; \node (f) at (1,0) {0};
\node (g) at (.6,-.4) {1};
\end{tikzpicture}},  \\
  $\beta_6$=\raisebox{-2.1ex}{\begin{tikzpicture}[scale= .85, transform shape]
\node (h) at (-.2,-.0) {0}; \node (a) at (0,0) {0}; \node (b) at (.2,0) {0}; \node (c) at (.4,0) {1};
\node (d) at (.6,0) {0}; \node (e) at (.8,0) {0}; \node (f) at (1,0) {0};
\node (g) at (.6,-.4) {0};
\end{tikzpicture}}=\,$\ap_4$, 
  $\beta_7$=\raisebox{-2.1ex}{\begin{tikzpicture}[scale= .85, transform shape]
\node (h) at (-.2,-.0) {0}; \node (a) at (0,0) {0}; \node (b) at (.2,0) {0}; \node (c) at (.4,0) {0};
\node (d) at (.6,0) {0}; \node (e) at (.8,0) {1}; \node (f) at (1,0) {0};
\node (g) at (.6,-.4) {0};
\end{tikzpicture}}=\,$\ap_6$,
  $\beta_8$=\raisebox{-2.1ex}{\begin{tikzpicture}[scale= .85, transform shape]
\node (h) at (-.2,-.0) {0}; \node (a) at (0,0) {0}; \node (b) at (.2,0) {0}; \node (c) at (.4,0) {0};
\node (d) at (.6,0) {0}; \node (e) at (.8,0) {0}; \node (f) at (1,0) {0};
\node (g) at (.6,-.4) {1};
\end{tikzpicture}}=\,$\ap_8$;
\end{description}

\noindent
If $\beta\in \eus K$ is a simple root, then $\Phi(\beta)=\{\beta\}$ and $\beta$ is necessarily
a minimal element of $\eus K$. Conversely, if $\beta$ is a
minimal element of $\eus K$ and $\Phi(\beta)=\{\ap\}$, a sole simple root, then $\ap=\beta$. This 
happens in all cases except $\GR{A}{2n}$. For instance, $\beta_n=\ap_n$ for $\GR{A}{2n-1}$, 
$\beta_4=\ap_3$ for $\GR{F}{4}$, 
$\beta_3=\ap_1$ for $\GR{E}{7}$, and $\beta_4=\ap_2$ for $\GR{E}{8}$.

\begin{figure}[htb]    
\caption{The marked cascade posets for  $\GR{A}{p}$ ($p\ge 2$), $\GR{C}{p}$ ($p\ge 1$), 
$\GR{F}{4}$}  
\label{fig:An}
\vskip1.5ex
\begin{center}
\begin{tikzpicture}[scale= .62] 
\node (2) at (4.5,6.3) {$\beta_1$};
\node (3) at (4.5,4.2) {$\dots$};
\node (4) at (4.5,2.1) {$\beta_{n-1}$};
\node (5) at (4.5,0) {$\beta_{n}$};
\foreach \from/\to in {2/3, 3/4, 4/5} \draw [-,line width=.7pt] (\from) -- (\to);

\draw (6.4,6.3) node {${\color{forest}\{} \un{\un{{\color{forest}\ap_{1}}}}, \un{\un{{\color{forest}\ap_{2n-1}}}}  {\color{forest}\}}$ };
\draw (7, 2.1) node {${\color{forest}\{}  \un{\un{ {\color{forest}\ap_{n-1}}}}, \un{\un{{\color{forest}\ap_{n+1}}}}  {\color{forest}\}}$ };
\draw (6,0) node {${\color{forest}\{} \un{\un{{\color{forest}\ap_n}}} {\color{forest}\}}$};

\draw (3,3.6) node {$\GR{A}{2n{-}1}$:};
\end{tikzpicture}
\ \ 
\begin{tikzpicture}[scale= .62] 
\node (2) at (4.5,6.3) {$\beta_1$};
\node (3) at (4.5,4.2) {$\dots$};
\node (4) at (4.5,2.1) {$\beta_{n-1}$};
\node (5) at (4.5,0) {$\beta_{n}$};
\foreach \from/\to in {2/3, 3/4, 4/5} \draw [-,line width=.7pt] (\from) -- (\to);

\draw (6.4,6.3) node {${\color{forest}\{} \un{\un{{\color{forest}\ap_{1}}}}, \un{\un{{\color{forest}\ap_{2n}}}}  {\color{forest}\}}$};
\draw (7.1, 2.1) node {${\color{forest}\{}  \un{\un{ {\color{forest}\ap_{n-1}}}}, \un{\un{{\color{forest}\ap_{n+2}}}}  {\color{forest}\}}$ };
\draw (6.7,0) node {${\color{forest}\{} \un{\un{ {\color{forest}\ap_{n}}}}, \un{\un{{\color{forest}\ap_{n+1}}}}  {\color{forest}\}}$ };

\draw (3,3.6) node {$\GR{A}{2n}$:};
\end{tikzpicture}
\ \ 
\begin{tikzpicture}[scale= .62] 
\node (2) at (4.5,6.3) {$\beta_1$};
\node (3) at (4.5,4.2) {$\dots$};
\node (4) at (4.5,2.1) {$\beta_{p{-}1}$};
\node (5) at (4.5,0) {$\beta_p$};
\foreach \from/\to in {2/3, 3/4, 4/5} \draw [-,line width=.7pt] (\from) -- (\to);

\draw (5.6,6.3) node {{\color{forest}$\{\ap_1\}$}};
\draw (6.2, 2.1) node {{\color{forest}$\{\ap_{p{-}1}\}$}};
\draw (5.6,0) node {${\color{forest}\{} \un{\un{{\color{forest}\ap_p}}} {\color{forest}\}}$};

\draw (3.3,3.6) node {$\GR{C}{p}$:};
\end{tikzpicture}
\ \
\begin{tikzpicture}[scale= .62] 
\node (2) at (4.5,6.3) {$\beta_1$};
\node (3) at (4.5,4.2) {$\beta_2$};
\node (4) at (4.5,2.1) {$\beta_3$};
\node (5) at (4.5,0) {$\beta_4$};
\foreach \from/\to in {2/3, 3/4, 4/5} \draw [-,line width=.7pt] (\from) -- (\to);

\draw (5.6,6.3) node {{\color{forest}$\{\ap_4\}$}};
\draw (5.6, 4.2) node {{\color{forest}$\{\ap_1\}$}};
\draw (5.6, 2.1) node {{\color{forest}$\{\ap_2\}$}};
\draw (5.6,0) node {{\color{forest}$\{\ap_3\}$}};

\draw (3.3,3.6) node {$\GR{F}{4}$:};
\end{tikzpicture}
\end{center}
\end{figure}

\begin{figure}[ht]    
\caption{The marked cascade posets for series  $\GR{B}{p}$, $p\ge 4$}  
\label{fig:Bn}
\vskip1.5ex
\begin{center}
\begin{tikzpicture}[scale= .62] 
\node  (1) at (9,10.5) { $\beta_1$};
\node  (2) at (6,8.4) { $\beta_2$};
\node  (3) at (9,8.4) { $\beta_3$};
\node  (4) at (6,6.3) { $\beta_4$};
\node  (5) at (9,6.3) {$\dots$};
\node  (6) at (6,4.2) {$\dots$};
\node  (7) at (9,4.2) { $\beta_{2n-3}$};
\node  (8) at (6,2.1) { $\beta_{2n-2}$};
\node  (9) at (9,2.1) { $\beta_{2n-1}$}; 
\node  (10) at (6,0) { $\beta_{2n}$}; 
\node  (11) at (9,0) { $\beta_{2n+1}$}; 
\foreach \from/\to in {1/2, 1/3, 3/4, 3/5, 5/7,7/8, 7/9, 9/10, 9/11} \draw [-,line width=.7pt] (\from) -- (\to);
\draw[loosely dotted] (5)--(6);

\draw (10.1,10.5) node {{\color{forest}  $\{\ap_2\}$}};
\draw (4.7,8.4) node {${\color{forest}\{} \un{\un{{\color{forest}\ap_1}}} {\color{forest}\}}$};
\draw (10.1,8.4) node {{\color{forest}  $\{\ap_4\}$}};
\draw (4.7, 6.3) node {{\color{forest}  $\{\ap_3\}$}};
\draw (10.9, 4.2) node {{\color{forest}  $\{\ap_{2n-2}\}$}};
\draw (4.1,2.1) node {{\color{forest}  $\{\ap_{2n-3}\}$}};
\draw (10.7,2.1) node {{\color{forest}  $\{\ap_{2n}\}$}};
\draw (4.3,0) node {{\color{forest}  $\{\ap_{2n-1}\}$}};
\draw (11,0) node {{\color{forest}  $\{\ap_{2n+1}\}$}};

\draw (2.6,4.2) node {$\GR{B}{2n+1}$:};
\end{tikzpicture}
\quad
\begin{tikzpicture}[scale= .62] 
\node (1) at (9,10.5) { $\beta_1$};
\node (2) at (6,8.4) { $\beta_2$};
\node (3) at (9,8.4) { $\beta_3$};
\node (4) at (6,6.3) { $\beta_4$};
\node (5) at (9,6.3) {$\dots$};
\node  (6) at (6,4.2)     {$\dots$};
\node (7) at (9,4.2) { $\beta_{2n-3}$};
\node (8) at (6,2.1) { $\beta_{2n-2}$};
\node (9) at (9,2.1) { $\beta_{2n-1}$}; 
\node (10) at (6,0) { $\beta_{2n}$}; 
\foreach \from/\to in {1/2, 1/3, 3/4, 3/5, 5/7,7/8, 7/9, 9/10} \draw [-,line width=.7pt] (\from) -- (\to);
\draw[loosely dotted] (5)--(6);

\draw (10.1,10.5) node {{\color{forest}  $\{\ap_2\}$}};
\draw (4.7,8.4) node {${\color{forest}\{} \un{\un{{\color{forest}\ap_1}}} {\color{forest}\}}$};
\draw (10.1,8.4) node {{\color{forest}  $\{\ap_4\}$}};
\draw (4.7, 6.3) node {{\color{forest}  $\{\ap_3\}$}};
\draw (10.9, 4.2) node {{\color{forest}  $\{\ap_{2n-2}\}$}};
\draw (4.1,2.1) node {{\color{forest}  $\{\ap_{2n-3}\}$}};
\draw (10.7,2.1) node {{\color{forest}  $\{\ap_{2n}\}$}};
\draw (4.3,0) node {{\color{forest}  $\{\ap_{2n-1}\}$}};

\draw (2.4,4.2) node { $\GR{B}{2n}$:};
\end{tikzpicture}
\end{center}
\end{figure}

\begin{figure}[ht]    
\caption{The marked cascade posets for series $\GR{D}{p}$, $p\ge 5$}   
\label{fig:Dn}
\vskip1.5ex
\begin{center}
\begin{tikzpicture}[scale= .62]
\node  (1) at (9,10.5) { $\beta_1$};
\node  (2) at (5.5,8.4) { $\beta_2$};
\node  (3) at (9,8.4) { $\beta_3$};
\node  (4) at (5.5,6.3) { $\beta_4$};
\node  (5) at (9,6.3) {$\dots$};
\node  (6) at (5.5,4.2) {$\dots$};
\node  (7) at (9,4.2) { $\beta_{2n-5}$};
\node  (8) at (5.5,2.1) { $\beta_{2n-4}$};
\node  (9) at (9,2.1) { $\beta_{2n-3}$}; 
\node  (10) at (5.5,0) { $\beta_{2n-2}$}; 
\node  (11) at (10,0) { $\beta_{2n-1}$}; 
\node  (12) at (12,0) { $\beta_{2n}$}; 
\foreach \from/\to in {1/2, 1/3, 3/4, 3/5, 5/7,7/8, 7/9, 9/10, 9/11,9/12} \draw [-,line width=.7pt] (\from) -- (\to);
\draw[loosely dotted] (5)--(6);

\draw (10.1,10.5) node {{\color{forest}  $\{\ap_2\}$}};
\draw (4.2,8.4) node {${\color{forest}\{} \un{\un{ {\color{forest}\ap_1}}}  {\color{forest}\}}$ };
\draw (10.1,8.4) node {{\color{forest}  $\{\ap_4\}$}};
\draw (4.2, 6.3) node {{\color{forest}  $\{\ap_3\}$}};
\draw (10.9, 4.2) node {{\color{forest}  $\{\ap_{2n-4}\}$}};
\draw (3.6,2.1) node {{\color{forest}  $\{\ap_{2n-5}\}$}};
\draw (10.9,2.1) node {{\color{forest}  $\{\ap_{2n-2}\}$}};
\draw (3.6,0) node {{\color{forest}  $\{\ap_{2n-3}\}$}};
\draw (8.4,0) node {${\color{forest}\{}\! \un{\un{{\color{forest}\ap_{2n}}}}\! {\color{forest}\}}$};
\draw (13.5,0) node {${\color{forest}\{} \un{\un{ {\color{forest}\ap_{2n-1}}}}  {\color{forest}\}}$ };

\draw (2.8,4.2) node { $\GR{D}{2n}$:};
\end{tikzpicture}
\quad
\begin{tikzpicture}[scale= .62]
\node (1) at (9,10.5) { $\beta_1$};
\node (2) at (6,8.4) { $\beta_2$};
\node (3) at (9,8.4) { $\beta_3$};
\node (4) at (6,6.3) { $\beta_4$};
\node (5) at (9,6.3) {$\dots$};
\node  (6) at (6,4.2)     {$\dots$};
\node (7) at (9,4.2) { $\beta_{2n-3}$};
\node (8) at (6,2.1) { $\beta_{2n-2}$};
\node (9) at (9,2.1) { $\beta_{2n-1}$}; 
\node (10) at (6,0) { $\beta_{2n}$}; 
\foreach \from/\to in {1/2, 1/3, 3/4, 3/5, 5/7,7/8, 7/9, 9/10} \draw [-,line width=.7pt] (\from) -- (\to);
\draw[loosely dotted] (5)--(6);

\draw (10.1,10.5) node {{\color{forest}  $\{\ap_2\}$}};
\draw (4.7,8.4) node {${\color{forest}\{} \un{\un{ {\color{forest}\ap_1}}}  {\color{forest}\}}$};
\draw (10.1,8.4) node {{\color{forest}  $\{\ap_4\}$}};
\draw (4.7, 6.3) node {{\color{forest}  $\{\ap_3\}$}};
\draw (10.9, 4.2) node {{\color{forest}  $\{\ap_{2n-2}\}$}};
\draw (4.1,2.1) node {{\color{forest}  $\{\ap_{2n-3}\}$}};
\draw (11.5,2.1) node {${\color{forest}\{}  \un{\un{ {\color{forest}\ap_{2n}}}}, \un{\un{{\color{forest}\ap_{2n+1}}}}  {\color{forest}\}}$ };
\draw (4.2,0) node {{\color{forest}  $\{\ap_{2n-1}\}$}};

\draw (2.9,4.2) node {$\GR{D}{2n+1}$:};
\end{tikzpicture}
\end{center}
\end{figure}

\vspace{.4cm}

\begin{figure}[ht]    
\caption{The marked cascade posets for $\GR{D}{4}$, $\GR{B}{3}$, $\GR{G}{2}$}   
\label{fig:D4,B3,G2}
\vskip1.5ex
\begin{center}
\begin{tikzpicture}[scale= .62]
\node  (9) at (9,2.1) { $\beta_{1}$}; 
\node  (10) at (6.2,0) { $\beta_{2}$}; 
\node  (11) at (9.5,0) { $\beta_{3}$}; 
\node  (12) at (12,0) { $\beta_{4}$}; 
\foreach \from/\to in {9/10, 9/11,9/12} \draw [-,line width=.7pt] (\from) -- (\to);

\draw (10,2.1) node {{\color{forest}  $\{\ap_{2}\}$}};
\draw (5.1,0) node {${\color{forest}\{} \un{\un{{\color{forest}\ap_1}}} {\color{forest}\}}$};
\draw (8.5,0) node {${\color{forest}\{} \un{\un{{\color{forest}\ap_4}}} {\color{forest}\}}$};
\draw (13.1,0) node {${\color{forest}\{} \un{\un{ {\color{forest}\ap_3}}}  {\color{forest}\}}$ };
\end{tikzpicture}
\qquad
\begin{tikzpicture}[scale= .62]
\node  (9) at (9,2.1) { $\beta_{1}$}; 
\node  (10) at (7.5,0) { $\beta_{2}$}; 
\node  (11) at (10.5,0) { $\beta_{3}$}; 
\foreach \from/\to in {9/10, 9/11} \draw [-,line width=.7pt] (\from) -- (\to);

\draw (10,2.1) node { {\color{forest}  $\{\ap_{2}\}$} };
\draw (6.4,0) node {${\color{forest}\{} \un{\un{{\color{forest}\ap_1}}} {\color{forest}\}}$};
\draw (9.5,0) node { {\color{forest}  $\{\ap_{3}\}$} };
\end{tikzpicture}
\qquad
\begin{tikzpicture}[scale= .62]
\node  (9) at (9, 2.1) { $\beta_{1}$}; 
\node  (10) at (9,0) { $\beta_{2}$}; 
\foreach \from/\to in {9/10} \draw [-,line width=.7pt] (\from) -- (\to);

\draw (10,2.1) node { {\color{forest}  $\{\ap_{2}\}$} };
\draw (10,0) node { {\color{forest}  $\{\ap_{1}\}$} };
\end{tikzpicture}
\end{center}
\end{figure}

\begin{figure}[ht]    
\caption{The marked cascade posets for $\GR{E}{6}$, $\GR{E}{7}$, $\GR{E}{8}$}  
\label{fig:En}
\begin{center}
\begin{tikzpicture}[scale= .63] 
\node (2) at (4.5,6.3) {$\beta_1$};
\node (3) at (4.5,4.2) {$\beta_2$};
\node (4) at (4.5,2.1) {$\beta_3$};
\node (5) at (4.5,0) {$\beta_4$};
\foreach \from/\to in {2/3, 3/4, 4/5} \draw [-,line width=.7pt] (\from) -- (\to);

\draw (5.6,6.3) node {{\color{forest}$\{\ap_6\}$}};
\draw (6.1, 4.2) node {${\color{forest}\{}  \un{\un{ {\color{forest}\ap_1}}}, \un{\un{{\color{forest}\ap_5}}}  {\color{forest}\}}$ };
\draw (6.1, 2.1) node {{\color{forest}$\{\ap_2,\ap_4\}$}};
\draw (5.6,0) node {{\color{forest}$\{\ap_3\}$}};

\draw (3,3.6) node {\large $\GR{E}{6}$:};
\end{tikzpicture}
\qquad
\begin{tikzpicture}[scale= .63] 
\node (2) at (10.5,6.3) {$\beta_1$};
\node (3) at (9,4.2) {$\beta_2$};
\node (4) at (10.5,2.1) {$\beta_3$};
\node (5) at (7.5,2.1) {$\beta_4$};
\node (6) at (6,0) {$\beta_5$};
\node (7) at (9,0) {$\beta_6$};
\node (8) at (11.5,0) {$\beta_7$}; 
\foreach \from/\to in {2/3, 3/4, 3/5, 5/6, 5/7, 5/8} \draw [-,line width=.7pt] (\from) -- (\to);

\draw (9.5,6.3) node {{\color{forest}$\{\ap_6\}$}};
\draw (8, 4.2) node {{\color{forest} $\{\ap_2\}$}};
\draw (11.5, 2.1) node {${\color{forest}\{} \un{\un{{\color{forest}\ap_1}}} {\color{forest}\}}$};
\draw (8.6,2.1) node {{\color{forest} $\{\ap_4\}$}};
\draw (5,0) node {{\color{forest} $\{\ap_3\}$}};
\draw (8,0) node {{\color{forest} $\{\ap_5\}$}};
\draw (10.5,0) node {{\color{forest} $\{\!\ap_7\!\}$}};

\draw (5,3.6) node {\large $\GR{E}{7}$:};
\end{tikzpicture}
\qquad
\begin{tikzpicture}[scale= .63] 
\node   (1) at (9,8.4) {$\beta_1$};
\node   (2) at (10.5,6.3) {$\beta_2$};
\node   (3) at (9,4.2) {$\beta_3$};
\node   (4) at (10.5,2.1) {$\beta_4$};
\node   (5) at (7.5,2.1) {$\beta_5$};
\node   (6) at (6,0) {$\beta_6$};
\node   (7) at (9,0) {$\beta_7$};
\node   (8) at (11.5,0) {$\beta_8$}; 
\foreach \from/\to in {1/2, 2/3, 3/4, 3/5, 5/6, 5/7, 5/8} \draw [-,line width=.7pt] (\from) -- (\to);

\draw (10,8.4) node {{\color{forest} $\{\ap_1\}$}};
\draw (11.5,6.3) node {{\color{forest} $\{\ap_7\}$}};
\draw (10, 4.2) node {{\color{forest} $\{\ap_3\}$}};
\draw (11.5, 2.1) node {{\color{forest} $\{\ap_2\}$}};
\draw (8.6,2.1) node {{\color{forest} $\{\ap_5\}$}};
\draw (7.1,0) node {{\color{forest} $\{\ap_4\}$}};
\draw (10.1,0) node {{\color{forest} $\{\ap_6\}$}};
\draw (12.5,0) node {{\color{forest} $\{\ap_8\}$}};

\draw (6,3.6) node {\large $\GR{E}{8}$:};
\end{tikzpicture}
\end{center}
\end{figure}


\begin{thebibliography}{P95}

\bibitem{ap97} {\sc D.V.~Alekseevsky} and {\sc A.M.~Perelomov}.
Poisson and symplectic structures on Lie algebras. I,
{\it J. Geom. Phys.}  {\bf 22}\,(1997),  no.\,3, 191--211.

\bibitem{cp3} {\sc P.~Cellini} and {\sc P.~Papi}.
Abelian ideals of Borel subalgebras and affine Weyl groups, 
{\it Adv. Math.}, {\bf 187}\,(2004), 320--361.

\bibitem{ag03} {\sc A.G.~Elashvili} and {\sc A.~Ooms}. 
On commutative polarizations, {\it J.~Algebra}  {\bf 264}\,(2003), 129--154.

\bibitem{jos76}   {\sc A.~Joseph}. 
The minimal orbit in a simple Lie algebra and its associated maximal ideal, 
{\it Ann. Sci. \'Ecole Norm. Sup.}, {\bf 9}\,(1976), no.\,1, 1--29.

\bibitem{jos77} {\sc A.~Joseph}. 
A preparation theorem for the prime spectrum of a semisimple Lie algebra,
{\it  J. Algebra}  {\bf 48}\,(1977), 241--289.

\bibitem{ko12} {\sc B.~Kostant}. 
The cascade of orthogonal roots and the coadjoint structure of the nilradical of a Borel subgroup of a semisimple Lie group, {\it Mosc. Math. J.}, {\bf 12}\,(2012), no.~3, 605--620.

\bibitem{ko13} {\sc B.~Kostant}. 
Center $\eus U(\n)$, cascade of orthogonal roots, and a construction of Lipsman--Wolf, ``Lie groups: structure, actions, and representations'', 163--173, Progr. Math., {\bf 306}, Birkh\"auser/Springer, New York, 2013.

\bibitem{lw} {\sc R.~Lipsman} and {\sc J.~Wolf}.
Canonical semi-invariants and the Plancherel formula for parabolic groups,
{\it Trans. Amer. Math. Soc.}, {\bf 269}\,(1982), 111--131.

\bibitem{ooms}  {\sc A.~Ooms}.
On certain maximal subfields in the quotient division ring of an enveloping algebra, 
{\it J. Algebra}, {\bf  230}\,(2000), no.~2, 694--712.

\bibitem{imrn} {\sc D.~Panyushev}. 
Abelian ideals of a Borel subalgebra and long positive roots, 
{\it Intern. Math. Res. Notices,\/} (2003), no.\,35, 1889--1913.

\bibitem{cambr} {\sc D.~Panyushev}.
The index of a Lie algebra, the centraliser of a nilpotent element,
and the normaliser of the centraliser, {\it Math. Proc. Camb. Phil. Soc.},
{\bf 134}, Part\,1 (2003), 41--59.

\bibitem{p06} {\sc D.~Panyushev}.
The poset of positive roots and its relatives, 
{\it J. Algebraic Combin.}, {\bf 23}\,(2006), 79--101.

\bibitem{mics}  {\sc D.~Panyushev}. 
Minimal inversion complete sets and maximal abelian ideals,
{\it J. Algebra}, {\bf 445}\,(2016), 163--180. 

\bibitem{p20a} {\sc D.~Panyushev}.
Abelian ideals of a Borel subalgebra and root systems, II,  
{\it Algebr. Represent. Theory}, {\bf 23}, no.\,4 (2020), 1487--1498. 

\bibitem{p20b} {\sc D.~Panyushev}.
Glorious pairs of roots and abelian ideals of a Borel subalgebra, 
{\it J. Algebraic Combin.}, {\bf 52}, no.\,4 (2020), 505--525.

\bibitem{sad}  {\rusc S.T.~Sad{e1}tov}. 
{\rus Dokazatel{\cprime}stvo gipotezy Miwenko--Fomenko},
{\rusi Doklady RAN}, {\bf 397}\,(2004), no.~6, 751--754. 
English translation: {\sc S.T.~Sadetov}. A proof of the Mishchenko--Fomenko conjecture, 
{\it Doklady Math.} {\bf 70}\,(2004), no.~1, 634--638.

\bibitem{vi90} 
{\rusc {E1}.B.\,Vinberg}. {\rus O nekotorykh kommutativnykh
podalgebrakh universal{\cprime}no{\u\i} obertyvayuwe{\u\i} algebry},
{\rusi Izv. AN SSSR. Seriya Matem.} {\bf 54}, {\rus N0}\,1 (1990), 3--25 (Russian).
English translation: 
{\sc E.B.\,Vinberg.} On certain commutative subalgebras of a
universal enveloping algebra, {\it Math. USSR-Izv.},
{\bf 36}\,(1991), 1--22.


\bibitem{VO}  {\rusc {E1}.B.~Vinberg, A.L.~Oniwik}. {\rusi
Seminar po gruppam Li i algeb\-rai\-qes\-kim gruppam}. {\rus Moskva: ``Nauka"}, 1988
(Russian). English translation: {\sc A.L.~Onishchik} and 
{\sc E.B.~Vinberg}. ``Lie groups and algebraic groups'', Berlin: Springer, 1990.

\bibitem{kos} {\sc O.~Yakimova}.
Commutative subalgebras of $\eus U(\q)$ of maximal transcendence degree, {\it Math. Res. Letters}, 
{\bf 28}\,(2021), no.~3, 907--924. 

\bibitem{yu}  {\sc R.W.T.~Yu}.
On the sum of the index of a parabolic subalgebra and of its nilpotent radical, 
{\it Proc. Amer. Math. Soc.} {\bf 136}\,(2008), no.~5, 1515--1522. 

\end{thebibliography}
\end{document}